\documentclass[a4paper,12pt]{article}
\usepackage{lscape}
\usepackage{a4,amssymb,amsmath,amsthm,url,lineno,threeparttable}
\usepackage[normalem]{ulem}
\usepackage{bezier,amsfonts,graphicx,latexsym,url}
\usepackage[utf8]{inputenc}
\usepackage[T1]{fontenc}
\usepackage[english]{babel}
\usepackage{color}
\usepackage{ulem,mathtools}
\usepackage{pst-plot}
\usepackage{float}

\title{Index of Parameters of Iterated Line Graphs}
\date{}
\begin{document}
\newtheorem{theorem}{Theorem}[section]
\newtheorem*{theoremnonum}{Theorem}
\newtheorem{definition}[theorem]{Definition}
\newtheorem{conjecture}[theorem]{Conjecture}
\newtheorem{proposition}[theorem]{Proposition}
\newtheorem{corollary}[theorem]{Corollary}
\newtheorem{lemma}[theorem]{Lemma}
\newtheorem{example}[theorem]{Example}
\newtheorem{remark}[theorem]{Remark}
\newtheorem{observation}[theorem]{Observation}
\newtheorem*{problem}{Problem}
\newtheorem{thmx}{Theorem}
\renewcommand{\thethmx}{\Alph{thmx}} 
\newcommand{\bal}{{\rm bal}}
\newcommand{\sbal}{{\rm sbal}}
\newcommand{\ot}{{\rm ot}}
\newcommand{\ex}{{\rm ex}}

\newcommand{\ind}{{\rm ind}}
\newcommand{\diam}{{\rm diam}}
\newcommand{\tot}{{\rm tot}}
\newcommand{\half}{{\rm Half}}

\newcommand{\ad}[1]{\textcolor{blue}{#1}}

\DeclareGraphicsExtensions{.pdf,.png,.jpg}

\author{Yair Caro \\ Department of Mathematics\\ University of Haifa-Oranim \\ Israel \and \ Josef  Lauri\\ Department of Mathematics \\ University of Malta \\ Malta \and Christina Zarb \\Department of Mathematics \\University of Malta \\Malta  }

\maketitle

\begin{abstract} \fontsize{11}{12}\selectfont

Let $G$ be a prolific graph, by which we mean a finite connected simple graph which is not isomorphic to a cycle nor a path nor the star graph $K_{1,3}$. The line-graph of $G$, denoted by $L(G)$, is defined by having its vertex-set equal to the edge-set of $G$ and two vertices of $L(G)$ are adjacent if the corresponding edges are adjacent in $G$. For a positive integer $k$, the iterated line-graph $L^k(G)$ is defined recursively by $L^k(G)=L(L^{k-1}(G))$.

In this paper we shall consider fifteen well-known graph parameters and study their behaviour when the operation of taking the line-graph is iterated. We shall first show that all of these parameters are unbounded, that is, if $P=P(G)$ is such a parameter defined on any prolific graph $G$, then $P(L^k(G)) \rightarrow \infty$ when $k \rightarrow \infty$. This idea of unboundedness is motivated by a well-known old result of van Rooij and Wilf that says that the number of vertices is unbounded if and only if the graph is prolific. 

Following this preliminary result, the main thrust of the paper will be the study of the value of  $k(P,{\mathcal F})$, which is the index of a family of prolific graphs with regards to a given graph parameter $P(G)$.  For a given parameter $P(G)$,  the index of $G$  is denoted by $\ind(P,G) = \min \{ r : P(G) < P(L^r(G) \}$.

Now for a family $\mathcal F$ of prolific graphs, the index of the family is $k(P,\mathcal{F}) = \max \{ \ind(P,G) :  G \in \mathcal F\}$, that is $k(P,F)$ is the smallest integer $k$ such that for every prolific graphs $G \in \mathcal F$, $\ind(P,G) \leq  k(P,\mathcal F)$.

The problem of determining the index of a parameter over the family of prolific graphs  is motivated by a  classical result of Chartrand who showed that it could require $k=|V(G)|-3$ iterations to guarantee that $L^k(G)$ has a hamiltonian cycle. 

For twelve of the fifteen parameters considered, we exactly determine $k(P,\mathcal F)$  where $\mathcal F$ is the family of all prolific graphs, and for some parameters we also characterize the class of prolific graphs realizing the extremal value $k(P,\mathcal F$).

For example, for the matching number $\mu$, we show that the index of every prolific graph is at most 4 which is sharp, namely $k(\mu,\mathcal F) = 4$ and we further characterize those graphs for which $\ind(\mu,G) = 4$.   

Interesting open problems remain, in particular completing the determination of $k(P,\mathcal F)$ for the three parameters:  the independence number, independent domination number and domination number, where we obtain partial results.

\end{abstract}

\section{Introduction}

The taking of the line-graph $L(G)$ of a graph $G$ is perhaps the most widely studied graph operation. This is probably because it is such a natural operation on graphs: a graph tells us which pairs of vertices are related by letting them be in the same 2-subset, and the line-graph takes this one step further by telling us which two subsets are related by containing a common vertex. The two fundamental theorems of line-graphs are arguably the characterisations by Krausz \cite{krausz1943demonstration} and Beineke \cite{beineke1970characterizations}. The latter's result is one of the most well-known characterisations of a class of graphs in terms of ``forbidden subgraph''. 

In this paper we shall study the effect of taking the line graph on fifteen well-known parameters. Our interest is in studying their behaviour when the operation of taking the line-graph is iterated. We denote by $L^k(G)$   the $k^{th}$ iterated line graph of $G$. In particular $L(G) = L^1(G),  L^2(G) = L(L(G))$ etc.  We consider prolific graphs --- a connected graph $G$ is called prolific  if $G$ is none of $P_k$, the path on $k$ vertices, $C_k$, the cycle on $k$ vertices or $K_{1,3}$.
Van Rooij and Wilf \cite{van1965interchange} consider the sequence $L^k(G)$ of iterated line graphs and show that, when $G$ is a finite connected graph, only four behaviours are possible for this sequence:
\begin{itemize}
	\item{if $G$ is a cycle graph then $L(G)$ and each subsequent graph in this sequence are isomorphic to $G$ itself. These are the only connected graphs for which $L(G)$  is isomorphic to $G$.}
	\item{if $G$ is a claw $K_{1,3}$, then $L(G)$ and all subsequent graphs in the sequence are triangles.}
	\item{if $G$ is a path graph then each subsequent graph in the sequence is a shorter path until eventually the sequence terminates with an empty graph.}
	\item{in all remaining cases, the sizes of the graphs in this sequence eventually increase without bound.}
\end{itemize}

The notion of boundedness of parameters follows naturally from the work of van Rooij and Wilf whose ideas were, thirty years later, also extended by Chartrand et al. \cite{chartrandHindex} to a modified form of the iterated line graph which they called the iterated $H$-line graph.

Formally, we define unbounded parameters as follows:
		\begin{itemize}
			\item{Let $G$ be a prolific graph.  A graph parameter $P(G)$  is \emph{unbounded for $G$}  if  $P(L^k(G)) \rightarrow \infty$  as $k \rightarrow \infty$. (Clearly $P(G)$ is unbounded for $G$ if and only if it is unbounded for $P(L^k(G))$  for $k \geq 0$.) }  
			\item{Let $\mathcal F$ be a family  of prolific graphs. A graph parameter $P(G)$  is \emph{unbounded on $\mathcal F$}  if  $P(L^k(G)) \rightarrow \infty$ as $k \rightarrow \infty$   for all members $G$ of $F$.}
			\item{A graph parameter is called \emph{unbounded} if it is unbounded on the family of all prolific graphs.}
		\end{itemize}

We first show that all of the fifteen parameters which we consider are unbounded. This is the result which one would intuitively expect, since 
$|V(L^k(G))|$ is unbounded. Yet, in some cases, we shall see that it does require some work to show unboundedness.

Following this, our attention then focuses not on the asymptotic behaviour of these parameters as the operation of taking the line-graph is iterated, but rather on their behaviour in the initial stages of this iterating process.  Formally, for a given parameter $P(G)$, the index of $G$  is denoted by $\ind(P,G) = \min \{ r : P(G) < P(L^r(G) \}$.

Now for a family $\mathcal F$ of prolific graphs, the index of the family is $k(P,\mathcal F) = \max \{ \ind(P,G) :  G \in \mathcal F\}$, that is $k(P,F)$ is the smallest integer $k$ such that for every prolific graphs $G \in \mathcal F$, $\ind(P,G) \leq k(P,\mathcal F)$.
In the case where this maximum over a family $\mathcal F$ or over all prolific graphs,  is finite, $P$ is called \emph{universal  over $\mathcal F$} or just \emph{universal} in case it is over all prolific graphs. Otherwise $P$ is non-universal and  we write $k(P,{\mathcal F}) =  \infty$, $k(P) = \infty$  respectively.

The motivation for this definition can best be described using Chartrand's result on the Hamiltonicity \cite{chartrand1973hamiltonian} of line-graphs. Chartrand showed that, for minimum degree $\delta$ at least 3,  there are non-Hamiltonian graphs $G$ such that $L^2(G)$ is still not Hamiltonian, but that $L^3(G)$ is hamiltonian for every such graph. Therefore if $P(G)$ denotes the number of hamiltonian cycles of $G$, then the index $k(P,\mathcal F)$ of $P$ over $\mathcal F$, the family of  all graphs with $\delta\geq3$ is 3. But for general prolific graphs Chartrand showed that it could require $n-3$ iterations to guarantee a Hamiltonian graph,and this is sharp, therefore the index  $k(P,\mathcal F)$  of $P$ over all prolific graphs is $\infty$.

Our aim is to find, for each of the fifteen parameters under consideration, their indices for as wide a family of prolific graphs as possible, and to characterise those extremal graphs which attain this value of the index $k(P,\mathcal F)$, that is, those $G$ with the property that $\ind(P,G)=k(P,\mathcal F)$.

 The notion of the index of a parameter was formally introduced for connectivity by Chartrand and Stewart in \cite{chartrand1969connectivity} but in the past fifty years it was studied by numerous authors for hamiltonicity \cite{chartrand1973hamiltonian}, matching number \cite{sumner1974graphs}, maximum and minimum degree \cite{hartke1999maximum},
connectivity \cite{KNOR2003255,SHAO20183441},  linkability \cite{bohme2006linkability}, maximally ordered graphs \cite{knor2006iterated}, independence number \cite{knor2006distance}, the independent domination number \cite{ALLAN197873}, more recently, the 1-crossing number \cite{wang}, to mention only a few. Other investigators also showed that there is interest in studying not only in the numerical behaviour of parameters for iterated line graphs but also qualitative properties such as planarity \cite{geblehplanarity} or, even more recently, generalized outerplanarity \cite{barati}.

Surprisingly, perhaps, the study of iterated line graphs has also found its way in the vast literature of the applications of graph theory to chemistry. For example, we can cite two very recent papers on the iterated line graphs and chemistry: in \cite{knorchem} the authors study the Wiener index of iterated line graphs, and in \cite{gutmanchem} the authors investigate the possible link between what are called the Bertz indices of the sequence of iterated line graphs and the study of quantitative structure-properties of molecules. It seems that the history of the relationship between line graphs and chemistry goes back a long time to a paper by Lennard-Jones and Hall in 1952!

However, in spite of all this interest, we are unaware of the existence of a comprehensive survey on  families of iterated line-graphs which collects the various results obtained in the last sixty years on both unboundedness of parameters and of their index. Such results are found in dozens of journals, so we have tried to collect those related to the fifteen parameters under consideration in this paper.
This paper therefore serves partly as a survey on this topic by collecting and presenting systematically results which are scattered in the literature, and partly as a presentation of new results in the hope that others might be interested in studying this aspect of iterated line graphs.  Throughout the paper, we follow general graph theory notation as in \cite{west2001introduction}.

\medskip\noindent
The parameters of prolific graphs which we shall consider are the following: \begin{enumerate}
	\item{$n(G)$,  the number of vertices of $G$}
	\item{$e(G)$,  the number of edges of $G$}
	\item{$\Delta(G)$, the maximum degree of $G$}                
	\item{$\delta(G)$, the minimum degree of $G$}                 
	\item{$d(G)$, the average degree of $G$}
	\item{$c(G)$, the longest cycle in $G$}                  
	\item{$\mu(G)$,  the matching number of $G$}
	\item{$\chi(G)$, the chromatic number of $G$}
	\item{$\chi '(G)$, the chromatic index of $G$}               
	\item{$\omega(G)$, the clique number of $G$}
	\item{$\lambda(G)$, the edge connectivity of $G$}   
	\item{$\kappa(G)$, the vertex connectivity of $G$}
	\item{$\alpha(G)$, the independence number of $G$}
	\item{$i(G)$,  the independent domination of $G$}   
	\item{$\gamma(G)$, the domination number of $G$}
\end{enumerate}

In order to give the reader a taste of what will be covered, we finish this introduction by giving a table which summarises the main results which we present in this paper. In this table, the class of all  prolific graphs is denoted by $\mathcal F$ and an asterisk by the name of a theorem means that part or all of the characterisation of the extremal graphs is still open. 


\bigskip
\begin{table}[h!]
	\caption{The fifteen unbounded parameters and their indices}
\begin{tabular}{|l|p{2.8cm}|p{1.5cm}|p{2.5cm}|}
	\hline
	Parameter   & Family of graphs & $k(P,\mathcal F)$& Theorem\\
	\hline
	Number of vertices $n(G)$  & $\mathcal F$ &  $4$  & Theorem B\\
	Number of edges $e(G)$     & $\mathcal F$ &  $2$  & Theorem A\\
	Maximum degree $\Delta(G)$ & $\mathcal F$ & $3$ & Theorem C\\
	Minimum degree $\delta(G)$  & ${\mathcal F}, \delta \leq 2$& $\infty$& Theorem D\\
	       & ${\mathcal F}, \delta\geq3$& $1$& Theorem D\\
	Average degree $d(G)$ &$\mathcal F$& 1 & Theorem E\\
    Longest cycle $c(G))$ & $\mathcal F$ & 1 & Theorem F\\
    Matching number $\mu(G)$ & $\mathcal F$ &4 & Theorem G\\
    Chromatic number $\chi(G)$ & $\mathcal F$ & 3 & Theorem H*\\
    Chromatic index $\chi'(G)$ & $\mathcal F$ & 3 & Theorem I\\
    Clique number $\omega(G)$& $\mathcal F$ & 3 & Theorem J\\
Edge connectivity $\lambda(G)$  & ${\mathcal F}, \delta \leq 2$& $\infty$& Theorem K\\
                                    & ${\mathcal F}, \delta\geq3$& $1$& Theorem K\\
    Vertex connectivity $\kappa(G)$ & ${\mathcal F}, \delta \leq 2$& $\infty$ & Theorem L\\
                                    & ${\mathcal F}, \delta\geq3$& $2$ & Theorem L*\\
        Independence number $\alpha(G)$ & ${\mathcal F}, d\geq4$& $\leq 2$& Theorem M*\\
							& ${\mathcal F},\delta\geq 3$& $\leq 2$& Theorem M*\\
							& ${\mathcal F},d\geq 3$& $\leq 3$& Theorem M*\\
                                    & ${\mathcal F}, \delta=2$& $\leq 3$& Theorem M*\\
    Independent domination number $i(G)$   & Open & &\\
    Domination number $\gamma(G)$ & ${\mathcal F}, \delta\geq 4$& $\leq 2$& Theorem N*\\ 
							& ${\mathcal F}, \delta=3$& $\leq 3$ & Theorem N*\\ 
							& ${\mathcal F}, d \geq 3 $& $ \leq 3$ & Theorem N*\\

	\hline
\end{tabular}	               -                                         -
\end{table}

 As regards notation,  we shall use $P(L^k(G)) = P_k(G)$ when no ambiguity is involved . For example, $\alpha_k(G) = \alpha(L^k(G))$ is the independence number of $L^k(G)$, the $k$-th iteration of $L(G)$.   Also, we use $x_j$ to denote the number of vertices of degree $j$ in a graph $G$.

\section{Unbounded Parameters}

We first state the following results which will be used in the main proof.

\begin{theorem}[Sumner  \cite{sumner1974graphs}] \label{matching}
Let $G$ be a connected $K_{1,3}$-free graph on $n$ vertices.  Then $\mu(G) = \lfloor \frac{n}{2} \rfloor$.
\end{theorem}

\begin{corollary} \label{ind_match}
If $G$ is a connected graph having $m$ edges then $\alpha_2(G)  =  \mu_1(G) =   \lfloor \frac{m}{2} \rfloor$.
\end{corollary}

This is found in \cite{knor2012independence}  and is obtained using an old result of Kotzig from 1957 in \cite{kotzig1957theory}.  An alternative proof is found in \cite{caro1980decompositions}.  However it is immediate from Sumner's theorem and the facts that  $n_1(G)= e(G)$  and $\mu_1(G) = \alpha_2(G)$.

 We now state the main theorem for this section.

\begin{theorem}
All the fifteen parameters listed above are unbounded.
\end{theorem}

\begin{proof}
 $\mbox{ }$
\begin{enumerate}

\item{For $n(G)$  the number of vertices of $G$, as already mentioned, it is known from \cite{van1965interchange} that $n(G)$ is unbounded  if $G$ is prolific. }
\item{For $e(G)$, the number of edges of $G$, the result follows directly from case 1, since the line graph of a prolific graph $G$ is  connected therefore $e_k\geq n_k -1$ for $k\geq 1$.}
\item{For $\Delta(G)$, maximum degree of $G$, the result can be deduced from the  theorem proved in \cite{hartke1999maximum}, which states that, for a prolific graph $G$,  there exists a constant $c(G)$ such that for $k \geq c(G)$,   $\Delta_{k+1} = 2\Delta_k - 2$,  and since $G$ is prolific $\Delta(G) \geq 3$ and hence $\Delta_k \rightarrow \infty$ when $k \rightarrow \infty$.}
\item{For $\delta(G)$, the minimum degree of $G$, this is proved explicitly in \cite{hartke2003minimum}, where it is shown that  for a prolific graph $G$,  there exists a constant $c(G)$ such that for $k \geq c(G)$, $\delta_{k} \geq 3$ and  $\delta_{k+1}= 2\delta_k - 2$.}
\item{For $d(G)$, the average degree of $G$, we know that $\Delta(G) \geq d(G)  \geq \delta(G)$ and due to item 4, we infer that $d(G)$ is unbounded.}
\item{For $c(G)$, the longest cycle in $G$, the result follows from a theorem of Chartrand proved in \cite{chartrand1968hamiltonian}, which states that for a prolific graph $G$ on $n$ vertices, $L^{n-3}(G)$  is Hamiltonian, and this fact together with item 1 prove that $c(G)$ is unbounded.

It is worth noting that if $\delta(G) \geq 3$,  then already $L^2(G)$ is Hamiltonian as proved in  \cite{chartrand1973hamiltonian}.}
\item{For $\mu(G)$, the matching number of $G$, we use item 2 and Theorem \ref{matching}, and also item 6, since if $G$ is Hamiltonian then  clearly $\mu(G)  =\lfloor \frac{n}{2} \rfloor$.}        
\item{For $\chi(G)$,  the chromatic number of $G$, it is clear that $\chi'(G) = \chi_1(G)$ hence by Vizing's theorem $\chi_{k+1}(G)   = \chi'_k(G) \geq \Delta_k(G)$, and the result follows from item 3.   }
\item{For $\chi'(G)$, the chromatic index of $G$, again by Vizing's theorem   $\chi'_k(G) \geq \Delta_k(G)$  and the result follows from item 3.}
\item{For $\omega(G)$, the clique number of $G$, clearly $\omega(L(G))=\Delta(G)$ and hence $\omega_{k+1}(G)  = \Delta_k(G)$ and the result follows from item 3.  }
\item{For $\lambda(G)$, the edge connectivity of $G$, there are results in \cite{KNOR2003255,SHAO20183441} which prove that for a prolific graph $G$,  there exists a constant $c(G)$  such that for $k \geq c(G)$, $\lambda_k(G) = \delta_k(G)$ and  the result follows from item 4.}
\item{ For $\kappa(G)$, the vertex connectivity of $G$, there are results in \cite{KNOR2003255,SHAO20183441} which proves that for a prolific graph $G$,  there exists a constant $c(G)$  such that for $k \geq c(G)$, $\kappa_k(G) = \delta_k(G)$, and  the result follows from item 4.}

\item{For $\alpha(G)$, the independence number of $G$, the result follows from Corollary  \ref{ind_match} and item 2.  A proof with details on the growth of $\alpha_k$ is given in \cite{knor2006distance}.}
\item{For $i(G)$, the independent domination number of $G$ we use a theorem by Allan-Laskar \cite{ALLAN197873}, which states that, for a $K_{1,3}$-free graph $G$ (hence for line graphs and iterations of line graphs), $i(G) = \gamma(G)$, and the proof of item 15.
}
\item{For $\gamma(G)$, the domination number of $G$, we observe the following.  Define $\mu^*(G)$ to be the minimum cardinality of a maximal matching in $G$.  Then
\begin{enumerate}
\item{it is well  known (Lemma 1 in   \cite{BIEDL20047})   that  $\mu^* \leq \mu \leq 2\mu^*$, hence $\mu^* \geq \frac{\mu}{2}$ and since by item 7 $\mu$ is unbounded, so is $\mu^*$.}
\item{it is clear that $i(L(G)) = \mu^*(G)$ \cite{van2020minimum}.  Hence $\gamma_{k+1} =  i_{k+1} = \mu^*_k$  so all parameters are unbounded since $\mu^*$ is unbounded. }
\end{enumerate}}
\end{enumerate}        

\phantom\qedhere
\end{proof}

In order to present the main  contribution of this paper, namely computing the  indices of the various parameters presented in section 2, we need several preparatory results, which are collected in the following section.

\section{Preparatory Tools}

We first consider some results and facts which we will use in our proofs.  These results mainly involve convexity and the well-known Jensen inequality \cite{jensen1906fonctions}.
\begin{theorem}[Jensen]
If $f$ is a real continuous function that is convex, then
\[f \left (\frac{\sum_{i=1}^n x_i}{n} \right ) \leq \frac{\sum_{i=1}^n f(x_i)}{n}.\]   Equality holds if and only if   $ x_1=x_2=\ldots=x_n$  or  if $f$ is a linear function on a domain containing $x_1, x_2,\ldots,x_n$.
\end{theorem}

The first known  and important result (unpublished \cite{federlimit}) which we shall use in further proofs concerns the average degree.
\begin{theorem} \label{avgdeg}
For a graph  $G$ we have $d_1(G)  \geq 2(d(G) - 1)$ and equality holds  if and only if $G$ is  regular, where d(G)   is the average degree of $G$.

 \end{theorem}

\begin{proof}

Clearly  $n_1(G)= e(G)  =  \frac{\sum \deg(v_ j )}{2}$   and  $2e(G) = nd(G)$.

Also $ e_1(G)  =  \sum \binom{\deg(v_j)}{2}$ and hence \[2e_1(G) =  2\sum \binom{\deg(v_j)}{2} \geq 2n \binom{d}{2}\]  by the Jensen inequality, since the function $\binom{x}{2}$ is convex.

Therefore \[d_1(G)=  \frac{2e_1(G)}{n_1(G)} \geq  \frac{2nd(G)(d(G)-1)}{ 2e(G)} =  \frac{2nd(G)(d(G)-1)}{nd(G)} =  2(d(G)-1)\] with  equality if and only $G$ is regular, again by the Jensen inequality.  

 \end{proof}

We now consider the difference between the number of edges.

\begin{theorem} \label{dreg}

For a graph $G$,
\[2(e_1(G) - e(G)) =  2\left (\sum \binom{\deg(v_j)}{2} - \frac{\sum \deg(v_j) }{2} \right)\]\[ = \sum \deg(v_j) (\deg(v_j)-2) \geq nd(G)(d(G)-2)\] with equality if and only if $G$ is regular.
\end{theorem}

 \begin{proof}

This result is clear using the Jensen inequality.

\end{proof}
We now consider the balanced degree sequence. First we give some notation.
We let 
$F(n,m)$  denote the set of all non-increasing non-negative integer sequences consisting of $n$ terms summing to $m$, and let $y=y(n,m)=y_1,\ldots y_n$ be the unique sequence in $F(n,m)$ such that $y_1-y_n\leq 1$.

\begin{theorem} \label{degsequence}

For any sequence $t=t_1 \ldots t_n \in F(n,m)$,
\[ \sum_{j=1}^n \binom{t_j}{2} \geq \sum_{j=1}^n \binom{y_j}{2} \geq n \binom{m/n}{2}\] 
and equality holds for the rightmost inequality if and only if $y_1=y_n$.

 \end{theorem}
\begin{proof}
The first inequality comes from the fact that  if $t_1 \geq t_n  +2$  then the sequence with $t_1^* = t_1 -1, \ldots, t_n^* = t_n +1$ (with reordering if necessary)  is again a member of $F(n, m)$ and has a strictly  smaller (triangular) sum, because it is an easy fact that 
$\binom{t_1}{2}+ \binom{t_n}{2} > \binom{t_1-1}{2}+\binom{t_n+1}{2}$.

The rightmost inequality is then the extremal case where all members of $(n,m)$ are equal  and  $d = \frac{ \sum_{j=1}^n y_j }{n}$ is  the common value of the members of $y(n,m)$.    

  \end{proof}

\begin{remark} \label{remark1}
We sometimes need to know the second best possible minimum sequence in $F(n,m)$, which we may need if $y(n,m)$ is not graphical or is realized only by non-prolific graphs.  We consider the following.  Suppose $n \geq 4 $ (otherwise there is no prolific graph),  and consider the sequence $y=y(n,m)$. Then,  either the value of $y_1$ repeats at least twice or the value of $y_n$ repeats at least twice.  Say $y_j = y_{j+1}$  for $j = 1$ or $j = n-1$.  Then replace  $y^*_j =  y_j+1$ and $y^*_{j+1} = y_{j+1} -1$  (for either $j = 1$ or $j = n-1$) and for the rest of the indices set $y^*_i = y_i$  and rearrange the sequence accordingly.  Let the obtained sequence be denoted $y^*=y^*(n,m)$,  then clearly  \[\sum\binom{y^*_i}{2}  = 1+  \sum\binom{y_i}{2}.\]  So the sequence $y^*(n, m)$ derived from $y(n,m)$ by  this switching  operation (in case $y(n,m)$ is not graphic or not realizable by any prolific graph) realizes the second best possible minimum after $y(n,m)$, as it differs by 1 and is a candidate to check if it is realizable by prolific graph.
\end{remark}

Lastly, we consider the line graph of a tree  and the difference in the number of edges.

 \begin{theorem} \label{linetree}

 If $T$ is a tree with $x_j$ vertices of degree $j$ then  \[2e_1(T) -2e(T) = - 2 + \sum_{j\geq 3} (j-1)(j-2)x_j.  \]

 \end{theorem}

\begin{proof}

In trees we have $x_ 1 =  2 + \sum_{j \geq 2} (j-2) x_j$, where $x_j$ is the number of vertices of degree $j$. 

Clearly  $2e(T) =  \sum \deg(v_j) =  \sum jx_j$  while  \[2e_1(G) =  \sum \deg(v_ j)(\deg(v_ j) - 1)  =  \sum j(j-1) x_j.\]

Hence \[2e_1(T) - 2e(T)  = \sum j(j-2)x_j  =   -x_1  +  \sum_{j \geq 2} j(j-2)x_j  =    - 2  -  \sum (j-2)x_j + \sum j(j-2)x_j  = \]\[- 2 + \sum_{j \geq 3} (j-1)(j-2)x_j .\]

 \end{proof}

\begin{theorem} \label{eGeH}

Let $G$ be a connected  graph  and $H$ a non-empty subgraph of $G$, then  $e_1(G) -e_1(H)\geq   e(G) - e(H)$.

\end{theorem} 

\begin{proof}

For $n = 2$  it follows that $G = H = K_2$ and this is trivially true. For  $n= 3$,  the only graphs for $G$ are  $P_3$ and $C_3$, and the only graphs for $H$ (ignoring isolated vertices) are $P_2$, $P_3$ and $C_3$ and the above holds.

Suppose $G$ is a minimum counter example with respect to $n= |V(G) |$, namely there exists non-empty subgraph $H$ of $G$ such that  $e_1(G) - e_1(H)  <   e(G) - e(H)$.  Clearly $n \geq 4$. Let $a_j   = \deg(v_ j)$  be the degree of $v_j$ in $G$  and $b_j$ be  the degree of $v_j$ in $H$ and observe that $a_j \geq b_j$.  Then for this counter example pair  we have \[2( e_1(G)  - e(G)) =    \sum a_j(a_j-2) <  \sum b_j( b_j - 2)  = 2( e_1(H)  - e(H)).\]

However this can happen only if $G$ contains a vertex $w$ of degree 1 which is not in $H$,  since otherwise  $a_j(a_j- 2)  \geq  b_j( b_j-2)$.

Define $G^* = G \backslash \{ w\}$. Then  $G^*$ is connected, $e(G^*) = e(G)- 1$ and $|V(G^*)| \geq 3$.  We apply the fact that $G$ is minimum counter example.  Then \[e_1(G^*) - e_1(H) \geq e(G^*) - e(H) =  e(G)- 1 - e(H).\]
So $e_1(G^*) \geq  e(G)- 1 + e_1(H)- e(H)$.  Hence if $e_1(G) \geq e_1(G^*) +1$ we are done .

Let $z$ be the neighbour in $G$ of the deleted leaf $w$. The leaf $w$ contributes nothing to $e_1(G)$ but since $G$ is connected and  $n \geq 4$, $\deg(z) \geq 2$, and in $G^* $, $\deg(z)$ decreases by 1 and is at least 1 so we have 
\[e_1(G)  -e_1(G^*)  = \binom{\deg(z)}{2}- \binom{\deg(z)- 1}{2} \geq 1\] and equality holds only if $\deg(w)  = 2$ and the claim is proved.

\end{proof}

Certain graphs play an important role in extremal cases of the  indices  of several parameters.  We describe and define the following:
\begin{definition}
A subdivision graph  of $K_{1,3}$  is called  type A, B, C  respectively if  the centre vertex of degree 3 is adjacent to 2, 1, 0 leaves respectively.  
\end{definition}

\noindent
These graphs are illustrated in Figure \ref{subdivK13}.

\bigskip
\begin{figure} [H]
	\begin{center}
%
%

\psscalebox{1.0 1.0} 
{
\begin{pspicture}(0,-2.0710464)(15.717115,2.0710464)
\psdots[linecolor=black, dotsize=0.2](2.5185578,0.80751127)
\psdots[linecolor=black, dotsize=0.2](1.7185577,0.007511268)
\psdots[linecolor=black, dotsize=0.2](2.5185578,0.007511268)
\psdots[linecolor=black, dotsize=0.2](3.3185577,0.007511268)
\rput[bl](2.3585577,1.0875113){$x$}
\rput[bl](0.7985577,1.8075112){\it Type $A$ subdivision}
\psdots[linecolor=black, dotsize=0.2](7.698558,0.76751125)
\psdots[linecolor=black, dotsize=0.2](6.8985577,-0.032488734)
\psdots[linecolor=black, dotsize=0.2](7.698558,-0.032488734)
\psdots[linecolor=black, dotsize=0.2](8.498558,-0.032488734)
\rput[bl](7.5385575,1.0475112){$x$}
\rput[bl](5.9785576,1.7675112){\it Type $B$ subdivision}
\psdots[linecolor=black, dotsize=0.2](13.278558,0.76751125)
\psdots[linecolor=black, dotsize=0.2](12.478558,-0.032488734)
\psdots[linecolor=black, dotsize=0.2](13.278558,-0.032488734)
\psdots[linecolor=black, dotsize=0.2](14.078558,-0.032488734)
\rput[bl](13.118558,1.0475112){$x$}
\rput[bl](11.5585575,1.7675112){\it Type $C$ subdivision}
\psline[linecolor=black, linewidth=0.04](2.4985578,0.7875113)(1.0785577,-0.5924887)(1.0785577,-0.5924887)
\psdots[linecolor=black, dotsize=0.2](1.0785577,-0.5924887)
\psline[linecolor=black, linewidth=0.04, linestyle=dotted, dotsep=0.10583334cm](1.0585577,-0.6324887)(0.13855769,-1.5724888)
\psdots[linecolor=black, dotsize=0.2](0.098557696,-1.6524887)
\psline[linecolor=black, linewidth=0.04](2.5385578,0.74751127)(2.5385578,0.04751127)(2.4985578,0.007511268)
\psline[linecolor=black, linewidth=0.04](2.5185578,0.80751127)(3.3385577,0.007511268)
\psline[linecolor=black, linewidth=0.04](7.678558,0.74751127)(6.9185576,-0.012488732)(6.8985577,-0.032488734)
\psline[linecolor=black, linewidth=0.04](7.718558,0.7875113)(7.718558,0.007511268)(7.658558,0.007511268)
\psline[linecolor=black, linewidth=0.04](7.718558,0.74751127)(8.478558,-0.012488732)
\psline[linecolor=black, linewidth=0.04](13.318558,0.7275113)(12.478558,-0.012488732)
\psline[linecolor=black, linewidth=0.04](13.258557,0.70751125)(13.258557,0.027511269)
\psline[linecolor=black, linewidth=0.04](13.298557,0.7275113)(14.0585575,-0.032488734)
\psline[linecolor=black, linewidth=0.04](6.8785577,-0.07248873)(6.258558,-0.6524887)(6.258558,-0.6524887)
\psdots[linecolor=black, dotsize=0.2](6.278558,-0.69248873)
\psline[linecolor=black, linewidth=0.04](7.718558,-0.052488733)(7.698558,-0.7724887)
\psdots[linecolor=black, dotsize=0.2](7.718558,-0.79248875)
\psline[linecolor=black, linewidth=0.04](12.438558,-0.07248873)(11.818558,-0.67248875)(11.818558,-0.67248875)
\psdots[linecolor=black, dotsize=0.2](11.818558,-0.67248875)
\psline[linecolor=black, linewidth=0.04](13.258557,-0.07248873)(13.238558,-0.73248875)(13.298557,-0.7724887)
\psdots[linecolor=black, dotsize=0.2](13.258557,-0.75248873)
\psline[linecolor=black, linewidth=0.04](14.158558,-0.052488733)(14.098557,-0.07248873)(14.718557,-0.67248875)
\psdots[linecolor=black, dotsize=0.2](14.718557,-0.69248873)
\psline[linecolor=black, linewidth=0.04, linestyle=dotted, dotsep=0.10583334cm](6.238558,-0.75248873)(5.3185577,-1.6924888)
\psdots[linecolor=black, dotsize=0.2](5.278558,-1.7724887)
\psline[linecolor=black, linewidth=0.04, linestyle=dotted, dotsep=0.10583334cm](7.738558,-0.85248876)(7.698558,-1.8524888)(7.698558,-1.8524888)
\psdots[linecolor=black, dotsize=0.2](7.698558,-1.9724888)
\psline[linecolor=black, linewidth=0.04, linestyle=dotted, dotsep=0.10583334cm](11.798557,-0.69248873)(10.878558,-1.6324887)
\psdots[linecolor=black, dotsize=0.2](10.838557,-1.7124888)
\psline[linecolor=black, linewidth=0.04, linestyle=dotted, dotsep=0.10583334cm](13.298557,-0.8324887)(13.258557,-1.8324888)(13.258557,-1.8324888)
\psdots[linecolor=black, dotsize=0.2](13.258557,-1.9524888)
\psline[linecolor=black, linewidth=0.04, linestyle=dotted, dotsep=0.10583334cm](14.778558,-0.73248875)(15.658558,-1.8724887)
\psdots[linecolor=black, dotsize=0.2](15.618558,-1.7924887)
\end{pspicture}
}
	\caption{Subdivision types of $K_{1,3}$}\label{subdivK13}
    \end{center}
\end{figure}
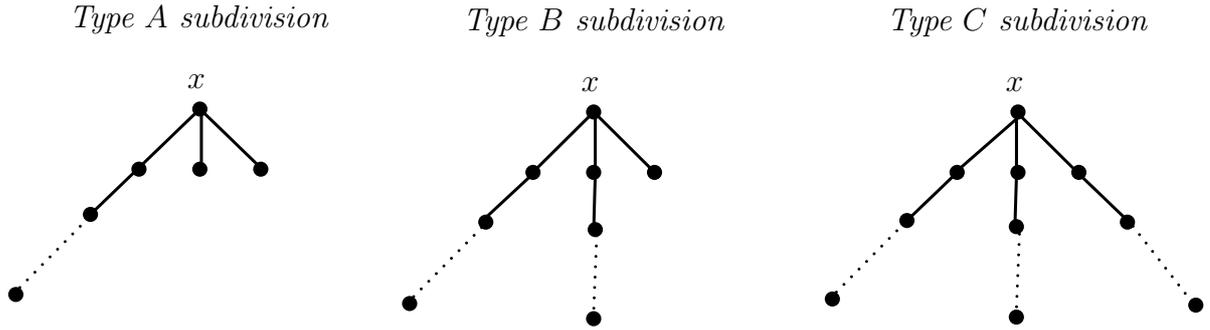

\noindent
We will also use the double star $S_{2,2}$ and its subdivisions as shown in Figure \ref{doublestar}. 

\bigskip
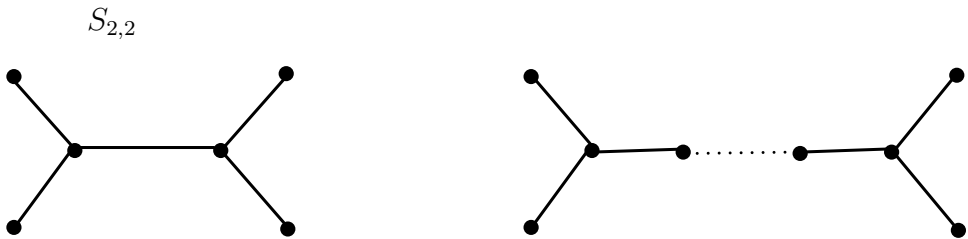
\begin{figure} [H]
	\begin{center}
%
%

\psscalebox{1.0 1.0} 
{
\begin{pspicture}(0,-1.4560465)(12.617115,1.4560465)
\psdots[linecolor=black, dotsize=0.2](0.098557696,0.68251127)
\psdots[linecolor=black, dotsize=0.2](0.098557696,-1.3174888)
\psdots[linecolor=black, dotsize=0.2](0.89855766,-0.29748872)
\psdots[linecolor=black, dotsize=0.2](2.8185577,-0.29748872)
\psdots[linecolor=black, dotsize=0.2](3.6785576,0.7225113)
\psdots[linecolor=black, dotsize=0.2](3.6985576,-1.3374888)
\psdots[linecolor=black, dotsize=0.2](6.8985577,0.68251127)
\psdots[linecolor=black, dotsize=0.2](6.8985577,-1.3174888)
\psdots[linecolor=black, dotsize=0.2](7.698558,-0.29748872)
\psdots[linecolor=black, dotsize=0.2](11.638557,-0.31748873)
\psdots[linecolor=black, dotsize=0.2](12.498558,0.70251125)
\psdots[linecolor=black, dotsize=0.2](12.518558,-1.3574888)
\psline[linecolor=black, linewidth=0.04](0.11855769,0.6025113)(0.8785577,-0.25748873)(0.11855769,-1.2974887)(0.11855769,-1.2974887)
\psline[linecolor=black, linewidth=0.04](0.9185577,-0.25748873)(2.7385576,-0.25748873)
\psline[linecolor=black, linewidth=0.04](3.6785576,0.68251127)(2.8385577,-0.27748874)
\psline[linecolor=black, linewidth=0.04](2.8385577,-0.31748873)(3.7185576,-1.3374888)
\psline[linecolor=black, linewidth=0.04](6.9185576,0.68251127)(7.718558,-0.25748873)
\psline[linecolor=black, linewidth=0.04](7.658558,-0.27748874)(6.9185576,-1.2974887)
\psline[linecolor=black, linewidth=0.04](11.658558,-0.31748873)(12.498558,0.7225113)
\psline[linecolor=black, linewidth=0.04](11.678557,-0.33748874)(12.518558,-1.3574888)
\psline[linecolor=black, linewidth=0.04](7.718558,-0.31748873)(8.918558,-0.27748874)
\psline[linecolor=black, linewidth=0.04](11.578558,-0.27748874)(10.458558,-0.31748873)
\psdots[linecolor=black, dotsize=0.2](8.898558,-0.31748873)
\psdots[linecolor=black, dotsize=0.2](10.438558,-0.33748874)
\psline[linecolor=black, linewidth=0.04, linestyle=dotted, dotsep=0.10583334cm](8.918558,-0.33748874)(10.398558,-0.31748873)(10.398558,-0.31748873)
\rput[bl](1.0585577,1.1825112){$S_{2,2}$}
\end{pspicture}
}
	\caption{$S_{2,2}$ and its subdivision type} \label{doublestar}
	\end{center}
\end{figure}

Another useful graph is $CP(k,n-k)$ which consists of a cycle and a pendant path as  illustrated in Figure \ref{CPnk}.
\begin{figure}[H]
	\begin{center}
%
%

\psscalebox{1.0 1.0} 
{
\begin{pspicture}(0,-2.0)(10.5971155,2.0)
\pscircle[linecolor=black, linewidth=0.04, dimen=outer](2.0985577,0.0){2.0}
\psdots[linecolor=black, dotsize=0.2](4.0985575,0.0)
\psdots[linecolor=black, dotsize=0.2](3.2985578,1.6)
\psdots[linecolor=black, dotsize=0.2](3.2985578,-1.6)
\psdots[linecolor=black, dotsize=0.2](0.89855766,-1.6)
\psdots[linecolor=black, dotsize=0.2](0.098557696,0.0)
\psdots[linecolor=black, dotsize=0.2](0.89855766,1.6)
\psline[linecolor=black, linewidth=0.04](4.0985575,0.0)(6.8985577,0.0)
\psline[linecolor=black, linewidth=0.04](8.898558,0.0)(10.498558,0.0)
\psdots[linecolor=black, dotsize=0.2](5.2985578,0.0)
\psdots[linecolor=black, dotsize=0.2](6.8985577,0.0)
\psdots[linecolor=black, dotsize=0.2](8.898558,0.0)
\psdots[linecolor=black, dotsize=0.2](10.498558,0.0)
\psline[linecolor=black, linewidth=0.04, linestyle=dotted, dotsep=0.10583334cm](6.8985577,0.0)(8.898558,0.0)
\rput[bl](1.2985576,0.0){$C_k$}
\rput[bl](7.2985578,-0.92){$n-k$}
\psline[linecolor=black, linewidth=0.04, arrowsize=0.033cm 2.0,arrowlength=1.4,arrowinset=0.0]{<->}(5.3385577,-0.52)(10.538558,-0.52)
\end{pspicture}
}
		\caption{The graph $CP(k,n-k)$: a path of length $n-k$ attached to a $k$-cycle} \label{CPnk}
	\end{center}
\end{figure}
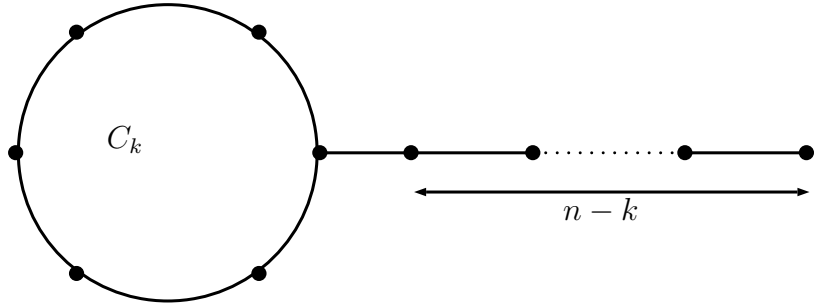

We state the following three lemmas first, then we give their proofs.

\begin{lemma} \label{typeC}
Let $G$ be a  type  C subdivision of $K_{1,3}$  with the three paths at the centre vertex $v$  having respectively $m_1 \geq m_2 \geq  m_3  \geq 2$ edges.  Then  
\[e_1(G) = e( G) \mbox{, } e_2(G)  =  e(G) +3 \mbox{, } e_3(G) \geq e( G) +15.\]
\end{lemma}

\begin{lemma} \label{typeB}
 
Let $G$ be a type B  subdivision of $K_{1,3}$  with the three paths at the centre vertex $v$  having respectively $m_1 \geq m_2 > m_3 =1$ edges.  Then

\[e_1(G) = e( G)  \mbox{, }  e_2(G) =  e(G) +2  \mbox{, } e_3(G) \geq e( G) + 9.\]

\end{lemma}

\begin{lemma} \label{typeA}
 
Let $G$ be a type A subdivision of $K_{1,3}$  with the three paths at the centre vertex $v$  having respectively $m_1 >m_2=  m_3 =1$ edges.  Then  

\begin{enumerate}
\item{if  $m_1  \geq 3$  then  $e_1(G) =  e(G) \mbox{, }   e_2(G) = e(G) + 1 \mbox{, } e_3(G) \geq e(G) +4.$}
\item{if $m_1 = 2$    then  $e_1(G) =  e(G) \mbox{, }   e_2(G)= e(G) + 1 \mbox{, } e_3(G) = e(G) +4.$}

\end{enumerate}
\end{lemma}
\medskip

 \emph{Proof of Lemmas \ref{typeC}, \ref{typeB} and \ref{typeA} }
\begin{itemize}
\item{Lemma \ref{typeC}: By Theorem \ref{eGeH}, it suffices to check  the case where $H$ is the type C subdivision graph with $m_1 =m_2 = m_3$ which gives the result.}
\item{Lemma \ref{typeB}: By Theorem \ref{eGeH}, it is suffices to check the case where $H$  is the type B subdivision with $m_1 = m_2 = 2$, $m_3 = 1$ which gives the result.}
\item{Lemma \ref{typeA}: By Theorem \ref{eGeH} it suffices to check the case where $H$ is a type A subdivision with $m_1 = 3$, $m_2 = m_3 = 1$ and  the case where $H$ is a type A subdivision with $m_1 = 2$, $m_2 = m_3  = 1$  which gives the result.  }
\end{itemize}

\begin{lemma} \label {prolific31}
Suppose $G$ is a prolific graph on $n$ vertices.  

\begin{enumerate}
\item{If  $e(G)  = n +2$ then  $e_1(G) \geq n +8$, namely $e_1(G) \geq e(G) +6$.   Equality is achieved  for all realizations of the graphic sequence with $n -4$  vertices of degree 2 and four vertices of degree 3. }   
\item{If $e(G)  = n +1$  then  $e_1(G) \geq n +4$, namely $e_1(G) \geq e(G) +3$.   Equality  is achieved for a cycle with one chord  or  for $n \geq 6$, in  two vertex-disjoint cycles of order $n_1\geq n_2 \geq 3$  with $n_1+n_2 = n$  connected  by a path with $n+1-(p_1+p_2)$ edges.   } 
\item{If  $e(G) = n$ then  $e_1(G)  \geq  n +1$, namely $e_1(G)  \geq e(G) +1$.  Equality  is achieved  for  a  cycle $C_k$ with an attached path on $n-k$ edges.    }
\end{enumerate}
\end{lemma}

 \begin{proof}
$\mbox{ }$\\
\begin{enumerate}
\item{By  Theorem \ref{degsequence}, the best possible sequence must be the sequence  $y(n,2n+4)$ and in this case it is precisely the sequence containing $n-4$ vertices of degree 2 and 4 vertices of degree 3. This is a graphical sequence with several realizations including cycles with two disjoint chords.}

\item{By  Theorem \ref{degsequence}, the best possible sequence must be the sequence  $y(n,2n+2)$ and in this case it is precisely the sequence containing $n-2$ vertices of degree 2 and 2 vertices of degree 3.  This is a graphical sequence and can be realized  by either a cycle  with a chord or, for $ n \geq 6$, also by   two vertex-disjoint cycles  of order $p_1\geq p_2 \geq 3$  connected by a path on $n+1-(p_1+p_2)$ edges. } 

 \item{By Theorem \ref{degsequence}, the best possible sequence must be the sequence  $y(n,2n)$ and in this case it is precisely the sequence containing $n$ vertices of degree 2 which forms a cycle $C_n$.  But $C_n$ is not prolific.  Hence the best second possible sequence, by Remark \ref{remark1},  is  $n-2$ vertices of degree 2, one vertex of degree 1 and one vertex of degree 3.  This second best possible sequence is graphically realized only by a cycle $C_k$ with an attached path on $n-k$ edges.    }
\end{enumerate}
\end{proof}

\begin{lemma} \label{prolific32}
Let $G$ be a prolific graph on $n$ vertices.
\begin{enumerate}

\item{Suppose  $e(G)  = n+k$  where $n \geq 2k$.  Then  $e_1(G) \geq  e(G)  +3k$ and   this is sharp.  }

 \item{Suppose $e(G)  \geq n +2$.  Then we have $e_1(G)  \geq  e(G) +6$.}
\end{enumerate}
\end{lemma}

 \begin{proof}
$\mbox{ }$\\

\begin{enumerate}
\item{Since $n\geq 2k$ we have   $2  < d(G) = \frac{2e(G)}{n} \leq \frac{2(n +k)}{n} \leq \frac{2(n+\frac{n}{2})}{n} \leq 3$, so equality holds only if $n = 2k$.  By Theorem \ref{degsequence}  the best possible lower bound is  $y(n,n+k)$ which must be precisely $n-k$ vertices of degree 2 and  $k$ vertices of degree 3.   This sequence is graphical and as $n$ increases with respect to $k$, many realizations exist including always the cycle $C_n$ with $k$ vertex –disjoint chords   (as $n \geq 2k$). }
\item{For $n = 4$ the condition is possible only for $K_4$ and the conclusion is true.  So we start induction on $n$. Suppose $n \geq 5$ and $e(G)  \geq n+3$ (otherwise it is already proved in Lemma \ref{prolific31} and Lemma \ref{prolific32} for $k=2$).  Then we can delete an edge $e$ such that $H = G-e$  is prolific, and $e(H) \geq n+2$.  Hence, by induction, using  the fact  that $H \subset  G$ implies $L(H) \subset  L(G)$, we get $e_1(G)\geq e_1(H) +1 \geq e(H) +7=e(G)+6$.  }
\end{enumerate}
\end{proof}

\section{Universal  Parameters}
\subsection{The number of edges $e(G)$:   $k(e,\mathcal F)=2$}
\begin{thmx}
Let $G$ be a prolific graph on $n \geq 4$ vertices.  Then
\begin{enumerate}
\item{$e_1(G) \geq e( G)$ with equality holding if and only if  $G$ is a subdivision of  $K_{1,3}$.}
\item{
 $e_1(G) \geq e(G) +1$ when $G$ is not a subdivision of $K_{1,3}$ and equality holds if and only if $G$ is  the double star $S_{2,2}$  or its subdivisions, or the graph $CP(k,n-k)$.}
\item{$e_2(G) \geq e(G) +1$ and  equality holds if and only if $G$ is a  type A subdivision  of $ K_{1,3}$ and  $e_2(G) = e_1(G) +1 = e(G) +1$.}   
\item{$ e_3(G) \geq  e(G) + 4$  with equality if and only if  $G$ is a type A subdivision  of $K_{1,3}$  on four edges, otherwise $e_3(G) \geq  e(G) +5$.}
\item{ $k(e,\mathcal F)=2$}
\end{enumerate}
\end{thmx}
\begin{proof}
$\mbox{ }$\\
\begin{enumerate}
\item{Since $G$ is connected it follows that $L(G)$  is connected  and $e_1(G) \geq e(G) - 1$  since $ e(G) = n_1(G)$.  If  $e_1(G) = e(G) - 1 = n_1(G) - 1$,  it means  $L(G)$  is a tree and this is possible only if  $G$ is a path (as a vertex of degree $d \geq 3$ creates a clique $K_d$ in $L(G)$), but then $G$ is not prolific.  So  we may assume $e_1(G) = e(G) = n(G) $. Then we have  $d_1(G) = 2$,  and if  $d(G) = 2$ we get  $d_1(G) = 2( d(G) - 1)$ which is possible by Theorem \ref{avgdeg}  if and  only if  $G$ is 2-regular which is not the case here since $G$ is prolific.  Hence  $2 >  2( d(G) - 1)$   and  $d(G)  < 2$  and hence, since $G$ is connected, $n(G)= e( G) +1$ and $G$ is a tree.  However, by  Theorem \ref{linetree},  we know that  \[2e_1(G) - 2e( G)  =   - 2 + \sum_{j\geq3} (j-1)(j-2)x_j.\]  
Now it is clear that,  if $x_3 \geq 2$ or, for some $j \geq 4$,  $x_j \geq 1$, we then have \[\sum_{j \geq 3}(j-1)(j-2)x_j  - 2 > 0\]contradicting $e_1(G) = e(G)$.  Hence  $2e_1(G) - 2e(G) = 0$  if and only if  $G$ is a tree with $x_3 = 1$  and $x_j  = 0$  for $j \geq 4$, hence $G $ is a subdivision of $K_{1,3}$.}

\item{Now suppose  $e_1(G) = e(G) +1 = n_1(G) +1$.  
\begin{enumerate}
\item{If $e(G) = n(G)-1$ then $G$ is a tree, but is not a subdivision of $K_{1,3}$ by part 1.  If, for some $j \geq 4$, $x_j \geq 1$,  then we get  \[2e_1(G) - 2e( G)  =  - 2 + \sum_{j\geq3} (j-1)(j-2)x_j \geq 4\] and  hence $e(L(G)) \geq e(G) +2$.  If  $x_3 \geq 3$  then we get  \[2e_1(G) - 2e( G)  =  -2 + \sum_{j \geq 3}(j-1)(j-2)x_j   \geq 4\] hence $e_1(G)  \geq e(G) +2$.  Hence, since $x_3 = 1$ this would mean that $G$ is a subdivision of $K_{1,3}$, and the remaining possibility is  that $G$ is a tree with $x_3 =2$  containing exactly two vertices of degree 3 hence $G$ is the double star $S_{2,2}$ (6 vertices, 5 edges) or its subdivisions, and indeed for such trees $e_1(G) = e(G) +1$.   }
\item{If $e(G) = n(G)$ then $d(G) = 2$ and since $G$ is prolific it is not 2-regular and  hence it is a unicyclic graph  which is not a cycle. Furthermore, by the  assumption in part 2, we have  $e_1(G) = e(G) +1  =  n(G) +1  = n_1(G) +1$.  Now, applying Theorem 3.3  \[2  =  2(e_1(G) - e( G) )  =  \sum \deg(v_j)(\deg(v_j) - 2)\] and we proceed as per item 1 with $x_j $ being the number of vertices of degree $j$  in $G$:
\begin{itemize}
\item{$\sum_{j \geq 1} x_j = n$. }
\item{$\sum_{j \geq 1} jx_j =  2n$}
\end{itemize}
 and hence subtracting the two equations we get  $\sum_{j\geq 1}(j-2)x_j = 0 $.  Hence  $x_1  =  \sum_{j \geq 2}(j-2)x_j$.  So we have \[ 2  =  2(e_1(G)-e( G) )  =  \sum \deg(v_j)(\deg(v_j) - 2)  = \sum_{j \geq 1} j(j-2)x_j\]   and substituting the value of $x_1$ we finally obtain  \[1 =  e_1(G) - e( G)   =  \sum_{j \geq 1} \binom{j-1}{2}x_j.\]  

Now if  for some $j \geq 4$,  $x_j \geq 1$, the right side is at least   3.  If  $x_3 \geq 2$, the right side is at least 2.  So we are left with $x_3 = 1$.   Since $G$ is a prolific graph with $e(G) = n(G)$, it must be unicyclic (but not a cycle) having exactly one vertex of degree 3, hence it must be a cycle $C_k$, $n >  k   \geq   3$  with a path $P_{n-k}$ attached. } 
\item{If $e(G)  \geq n(G) +1$   then \[e_1(G) = e(G) +1   \geq  n(G) +2.\]  But then \[d(G)  = \frac{2e(G)}{n(G)} =  \frac{2n_1(G)}{n(G)}\]  and \[\frac{2e_1(G)}{n_1(G)}=   d_1(G) \geq 2(d(G) - 1) > d(G)  =    \frac{2n_1(G)}{n(G)}.\]  From this we deduce the following:
\[e(G)^2 -1  =  (e(G) +1)(e(G) -1) =  e_1(G)(e(G)- 1) \]\[ \geq e_1(G)n(G) > (n_1(G))^2  =  e(G)^2\] a contradiction.  Hence if $e(G) \geq n(G) +1$, the assumption   $ e_1(G) = e(G) +1 = n_1(G) +1$ cannot also hold.   }

  \end{enumerate}}

  \item{If $G$ is not a subdivision of $K_{1,3}$  then by part 1, $e_1(G) > e(G)$  and since $ L(G)$  is not a subdivision of $K_{1,3}$ we have, again by part 1, $e_2(G) > e_1(G) > e(G)$ and  hence $e_2(G) \geq e( G) +2$.  If $G$ is a subdivision of $K_{1,3}$  then by part 1, $e_1(G) =  e(G)$. But  $L(G)$ itself is not a subdivision of $K_{1,3}$  (as it contains  $K_3$ formed by the edges incident with the vertex of degree 3) nor a cycle,  hence by part 1, $e_2(G) > e_1(G)  = e(G)$ .  Now checking subdivisions of $K_{1,3}$  according to the three types of paths having the root at $v$ of degree 3, we have by Lemmas \ref{typeC} - \ref{typeA}
\begin{itemize}
\item{if $G$ is a type C subdivision, then $e_2(G) =  e_1(G) +3 = e(G) +3$.}
\item{if $G$ is a type  B subdivision, then $e_2(G) =  e_1(G) +2 = e(G) +2$.}
\item{if $G$ is a type A subdivision, then  $e_2(G) = e_1(G) +1 = e(G) +1$.}
\end{itemize}}   
\item{ If $e_2(G) \geq  e_1(G)+2$,  then, by Lemma \ref{prolific31},  we have \[e_3(G) \geq n_2(G) +6  = e_2(G) +6  \geq e_1(G) +8   \geq  e(G) +7.\]

So we assume $e_2(G)  = n_2(G) +1 = e_1(G) +1$, and we can use item 2 with $H =L(G)$  and  $L^2(G) = L(H)$.  But this implies that $H = L(G)$  is either a subdivision of $S_{2,2}$ or the graph $CP(k,n-k)$,  and of these graphs only  $CP(3,n-3)$ is a line graph of a type A subdivision of $K_{1,3}$, and now by Lemma \ref{typeA} we are done.    

Lastly, we assume $e_2(G) = n_2(G) = e_1(G)$ and  hence, by part 1, $L(G)$  is a  subdivision of $K_{1,3}$ which is impossible.   }
\item{Now $k(e,\mathcal F)=2$ follows directly from the steps above.}
\end{enumerate}
\end{proof}

 \subsection{The number of vertices n(G):  $k(n,\mathcal F)=4$}

\begin{thmx}
Suppose $G$ is a prolific graph on $n \geq 4$  vertices and $e(G)$ edges.
\begin{enumerate}
\item{If $e(G) \geq n +1$,  then $n_1(G ) >  n(G).$}
\item{If $e(G) = n$,  then $n_2(G) > n(G).$}
\item{If  $e(G) = n-1$,  then $G$ is a tree and
\begin{enumerate}
\item{if $x_j  >  0$ for some $j \geq 4$,  then $n_2(G)  > n(G)$.}
\item{if $x_3 \geq 3$,  then  $n_2(G)  > n(G)$.}
\item{if $x_3 = 2$,  then  $n_2(G) = n(G)$ but $n_3(G)  > n(G)$  and $G$ is $S_{2,2}$ or a subdivision of $S_{2,2}$.}
\item{if $x_3 = 1$  and $G$ is a type B or type C subdivision of $K_{1,3}$,  then $n_3(G)  > n(G)$.}
\item{if $x_3 = 1$ and $G$ is a type A subdivision of $K_{1,3}$, then  $n_4(G)  > n(G)$.}
\end{enumerate}}
\item{ $k(n,\mathcal F)=4$.}
\end{enumerate}

\end{thmx}  

\begin{proof}
$\mbox{  }$\\
\begin{enumerate}
\item{Clearly $n(L(G)) = e(G) \geq n +1  > n(G)$.}
\item{If $n = e(G)$ then $n_1(G)= e(G) = n = n(G)$  but since $G$ is prolific it is not a cycle and hence  by Theorem \ref{avgdeg} $d_1(G)  > 2$  and \[e_1(G) = \frac{d_1(G)n_1(G)}{2}  >  \frac{2n_1(G)}{2} = n_1(G) = n.\]  Hence $e_1(G) = n_2(G)> n(G)$.}     
\item{Since $G$ is a tree we use Theorem \ref{linetree},   $2e_1(G) - 2e(G) = - 2 + \sum_{j \geq 3}(j-1)(j-2)x_j$.  As $G$ is not a path, $x_j > 0$  for some $j \geq 3$.

\begin{enumerate}
\item{if $x_j > 0$ for some $j \geq 4$   then  \[2e_1(G) - 2e(G) =-2+ \sum_{j \geq 3}(j-1)(j-2)x_j  \geq  -2  +6  = 4,\] hence $e_1(G)-e(G) = n_2(G) -  n +1 \geq 2$.  And we get $n_2(G) \geq n +1> n(G)$   and we are done.  So we may assume $x_j  = 0$  for $j \geq 4$.}
\item{if $x_3 \geq 3$  then  \[2e_1(G) -2e(G) = - 2 +\sum_{j \geq 3}(j-1)(j-2)x_j  \geq  - 2  +6  = 4,\] hence $e_1(G)-e(G) = n_2(G) -  n +1 \geq 2$.  And we get $n_2(G) \geq n +1> n(G)$   and we are done.}
\item{if $x_3 = 2$  then $G$ is $S_{2,2}$ or obtained by subdivisions from $S_{2,2}$ and \[2e_1(G) -2e(G) =  - 2 +\sum_{j \geq 3}(j-1)(j-2)x_j  = - 2  +4  = 2,\] hence $e_1(G)-e(G) = n_2(G) -  n +1  = 1$.  And we get $n_2(G) =  n  =  n(G)$.  But then $e_2(G)  = \frac{d_2(G)n_2(G)}{2} =  \frac{d_2(G)n}{2}$.  Also  $d_2(G) >  2(d_1(G) - 1)$ as $L(G)$  is not regular.  Hence \[d_2(G)> 2(\frac{2e_1(G)}{n_1(G)}- 1)  = 2 ( \frac{2n}{n-1} - 1)  = \frac{2(n+1)}{n-1}.\]

Hence $n_3(G) = e_2(G) > \frac{2(n+1)n}{2(n-1)}  > n$  and $n_3(G) \geq n +1 >  n(G)$.}
\item{if $x_3 = 1$ then $G$ is a subdivision of $K_{1,3}$. 

If $G$ is a type C subdivision of $K_{1,3}$ then by Lemma \ref{typeC}, $n_3(G) = e_2(G) = e(G) +3  = n +2$ and we are done.  

 \medskip

If $G$ is a type B subdivision of $K_{1,3}$ then by Lemma \ref{typeB},  $n_3(G) = e_2(G) = e(G) +2  =  n +1$ and we are done.}
\item{if $G$ is a  type A subdivision of $K_{1,3}$ then by Lemma \ref{typeA}, $n_3(G) = e_2(G)  =  e(G) +1 = n$, but  $n_4(G) = e_3(G) \geq e(G) +4  \geq n +3$  and we are done.}

\end{enumerate}}
\item{ Now $k(n,\mathcal F)=4$ follows directly from the steps above.}
 
\end{enumerate}
\end{proof}

\subsection{The maximum degree $\Delta(G)$:  $k(\Delta,\mathcal F)=3$}

\begin{definition}
We say that a  connected graph $G$ is \emph{fine} if it contains an edge $e =uv$  such that  $\deg(u) + \deg(v) - 2  > \Delta(G)$.
\end{definition}

\begin{thmx} \label{maxdeg}

Let $G$ be a prolific graph with maximum degree $\Delta \geq 3$.  Then
\begin{enumerate}
\item{$\Delta_1(G) > \Delta(G)$ if and only if $G$ is a fine graph.}
\item{if $G$ is not a fine graph and  $\Delta \geq 4$,  then   $\Delta_2(G) > \Delta(G)$ unless $G= K_{1,4}$  where $\Delta_3(G)= 6 > \Delta_2(G) = \Delta(G) = 4>\Delta_1(G) = 3$.}
\item{if $G$ is not a fine graph and  $\Delta = 3$, then $\Delta_2(G) > \Delta(G)$  unless $G$ is either a type A subdivision of $K_{1,3}$  or a tree obtained  from $S_{2,2}$ by subdividing the middle edge at least twice.}
\item{$k(\Delta,\mathcal F)=3$.}
\end{enumerate}
\end{thmx}
\begin{proof}
$\mbox{  }$\\
\begin{enumerate}

\item{This follows from the fact that \[\Delta_1(G) =  \max \{ \deg( u)  + \deg(v) -2 \mbox{, } uv \in E(G)\}.\]}
\item{We assume  $\Delta \geq 4$.  Let $v$ be a vertex of maximum degree $\Delta$.   Then all the neighbours of $v$ are of degree 1 or 2 otherwise $G$ is fine.   

If all neighbours of $v$ have degree 1, then $G$  is the star $K_{1, \Delta}$ and  $L(G) =  K_{\Delta}$ which is regular of degree  $\Delta-1$ and hence $\Delta_2(G) = 2\Delta- 4 > \Delta$ for $\Delta >4$.

So we need to consider $G =K_{1,4}$ with $L(G) = K_4$, $L^2(G) = K_{2,2,2}$ where $\Delta_2(G)  = 4$  and $\Delta_3(G)= 6 > \Delta(G)$.

If at least one neighbour of $v$ has degree 2, let us consider the vertices in $L(G)$  formed by the $\Delta$ edges incident with $v$.  These edges form $K_{\Delta}$ in $L(G)$ and since at least one of these  edges is incident with a vertex of degree 2 in $G$, it must have degree $\Delta$ in $L(G)$.

Consider an edge $e^*$  in this $K_{\Delta}$ incident with the vertex of degree $\Delta$ and another vertex of degree at least $\Delta -1$.  The vertex $e^*$ in $L^2(G)$ corresponding to the edge $e^*$ in $L(G)$  has degree  $2\Delta -3  > \Delta$ for $\Delta \geq 4$ and we are done.    }
\item{Suppose $\Delta(G) = 3$, recall  that we may assume that  the vertices of degree 3  are non-adjacent as all their neighbours are of degree 1 or 2 (by item 1), and at least one neighbour is of degree 2 (otherwise $G = K_{1,3 }$).

Suppose $\deg(u)  = 3$ and $v$, $w$ are neighbours of $u$ of degree 2 and $z$ is the third neighbour.  Then, in $L(G)$, the vertices representing the edges  $uv$ and $uw$  are adjacent vertices of degree 3  and hence the edge between them has degree 4 in $L^2(G)$.  So in this case $\Delta_2(G)=4> \Delta_1(G)=\Delta(G)=3$.

So we consider the case in which all vertices of degree 3 have exactly one neighbor of degree 2 and two leaves. This forces $G$ to be a tree  because consider a cycle in $G$.  Since $G$ is prolific there must be a vertex not on the cycle adjacent to a vertex $v$  on the cycle.  But then $v$ is of degree  3 and is adjacent to two vertices of degree 2 on the cycle, a contradiction.

Hence either $G$ is a tree with exactly one vertex of degree 3 with two leaves and a path on at least two edges starting from $u$, which is a type A subdivision of $K_{1,3}$ or $G$ is a tree having exactly two vertices of degree 3,  $u$ and $v$, with distance between these vertices at least 2.  Observe $G$ cannot have three such vertices of degree 3.

If  $G$ is a type A subdivision of $K_{1,3}$  we have   $\Delta_2(G) = \Delta_1(G) = 3$,  while $\Delta_3(G)= 4$.

Otherwise if G has two vertices $u$, $v$ of degree 3  with distance between these vertices exactly 2 then $\Delta_1(G) = 3$, but $L(G)$  contains two adjacent vertices of degree 3 hence $\Delta_2(G)  = 4$.  

On the other hand  if the distance between these vertices is at least 3  then $\Delta_2(G) = \Delta_1(G) = 3$ but $L^2(G)$ contains adjacent vertices of degree 3 hence $\Delta_3(G) = 4$.

Lastly, since in this case the vertices of degree 3 can have only one adjacent vertex of degree 2 it must be that $G$ is obtained from $S_{2,2}$ by subdividing the middle edge at least twice.}
\item{Now $k(\Delta,\mathcal F)=3$ follows directly from the steps above.}
\end{enumerate}
\end{proof}

\subsection{The minimum degree $\delta(G)$:   $k(\delta,\mathcal F)=\infty$ if $\delta=1,2$, otherwise $k(\delta,\mathcal F)=1$}

\begin{thmx}

Let $G$ be a prolific graph on $n \geq 4$ vertices. Then
\begin{enumerate}
\item{Let $F=\{G: \delta(G) \in \{1,2\}\}$.  Then  $k(\delta,\mathcal F)=\infty$.}
\item{Let $F=\{G: \delta(G)\geq 3\}$.  Then $k(\delta,\mathcal F)=1$.}
\end{enumerate}
\end{thmx}
\begin{proof}
$\mbox{ }$\\
\begin{enumerate}
\item{For $\delta=1$, consider the  prolific graph $G$ which contains an arbitrary long path say  of length $k$, then we need  $k$ iterations of the line graph before the vertex of degree 1 disappears. 

For $\delta=2$  consider the prolific graph  $G$ which contains two vertices $u$  and $v$  joined by a  path of length $k$ say, and   we need  $ \lceil \frac{k}{2} \rceil $ iterations of the line graph before the vertices of degree 2 of this path disappear.}
\item{For $\delta \geq 3$ we clearly have $\delta_1(G) \geq 2\delta(G) -2 > \delta(G)$.}
\end{enumerate}
\end{proof}

\subsection{The average degree $d(G)$:  $k(d,\mathcal F)=1$}

\begin{thmx}

Let $G$ be a prolific graph on $n \geq 4$ vertices.  Then $d_1(G) >d(G)$, that is  $k(d,\mathcal F)=1$.

\end{thmx}

\begin{proof}
By Theorem \ref{avgdeg},  $d_1(G) \geq  2(d(G) - 1)$ with equality if and only if $G$ is  regular.  Since $G$ is prolific, it is either a tree  with $d(G) =  2 - 2/n$ or non-regular with $d(G)  \geq 2$.

If $G$ is a tree,  $L(G)$ is connected and contains a cycle hence $d_1(G) \geq 2 > d(G)$.  In the case $d(G) \geq 2$  and $G$ is not 2-regular the result follows directly from Theorem \ref{avgdeg}.
\end{proof}

\subsection{The longest cycle $c(G)$ in $G$:   $k(c,\mathcal F)=1$}

\begin{thmx}

Let $G$ be a prolific graph on $n \geq 4$ vertices.  Then $c_1(G) >c(G)$,  that is $k(c,\mathcal F)=1$.

\end{thmx}

\begin{proof}
If $G$ is Hamiltonian then $L(G)$  is Hamiltonian  and since $G$ is prolific it follows that $e(G) > n(G)$  which implies that $n_1G) > n(G)$ and hence $c_1(G) > c(G)$.

If $G$ is a tree then clearly $G$ has no cycle but as $G$ is prolific it must contains a vertex of degree at least 3, and hence $L(G)$  contains a cycle and we are done.

So assume $G$ is such that  $c(G) =  q$, $3 \leq q \leq n-1$.  Since $G$ is connected there is a vertex $w$ not on the longest cycle but adjacent to a vertex $v$  on the longest cycle with  edges $xv$ and $yv$ being the edges  incident to $v$ on the longest cycle.  Then in $L(G)$  the longest cycle (now the edges of $G$ are represented by vertices in $L(G)$) is extended  by  a vertex representing the edge $wv$ and  hence $c_1(G)   > c(G)$.  
\end{proof}  

\subsection{The matching number $\mu(G)$:   $k(\mu,\mathcal F)=4$}

 
We first need a definition.
\begin{definition}
A graph $G$ on $n$ vertices is said to have a 1-factor/near 1-factor  if  $\mu(G) = \left \lfloor \frac{ n}{2 } \right \rfloor$.
\end{definition}

Hence referring to Theorem \ref{matching}, we can say that if $G$ is a connected  $K_{1,3}$-free graph, then $G$ has a 1-factor/near 1-factor accordingly with $n \equiv 0,1 \pmod 2$.

Also, for a connected graph  $G$, $\mu_1 =  \left \lfloor \frac{ e(G)}{2 } \right \rfloor$.  This follows directly from Theorem \ref{matching} since $n_1 = e(G)$  and line graphs are in particular $K_{1,3}$-free graphs.

First we need several lemmas dealing with the  index of the matching number $\mu$ and the  cases where  $G$ has $n$ vertices and $m$ edges where $m \geq n +2$,  $m = n +1$, $m  = n$  and $m = n-1$.

\begin{lemma}

Let $G$  be a prolific graph on $n$ vertices and $m$ edges such that $m \geq n +2$.  Then  $\mu_1 \geq \mu+1$  with equality if and only if $n \equiv 1\pmod 2$, $m = n +2$  and $G$ has a near 1-factor,  or $n \equiv 0\pmod 2$  and $m \in  \{ n +2,n+3\}$ and $G$ has a 1-factor.
\end{lemma}
 
\begin{proof}
If $n \equiv 1 \pmod 2$ and $m \geq n +3$, then  $n_1 \geq n +3 \equiv 0 \pmod 2$  and $\mu_1  \geq \frac{n+3}{2}   = \frac{n-1}{2} +2   \geq  \mu(G) +2$.  

If $n \equiv 1 \pmod 2$ and $m =n +2$, then  $n_1 = n +2 \equiv 1 \pmod 2$  and $\mu_1  =\frac{n+1}{2}   = \frac{n-1}{2} +1  \geq  \mu(G) +1$, where the last inequality holds as equality  if and only if $G$ has near 1-factor.

If $n \equiv 0\pmod 2$ and $m \geq n +3$, then  $n_1 \geq n +3 \equiv 1 \pmod 2$  and $\mu_1  \geq \frac{n+2}{2}   = \frac{n}{2} +1  \geq  \mu(G) +1$, where the last inequality holds as equality  if and only if $G$ has a 1-factor and $m=n+3$.

If $n \equiv 0 \pmod 2$ and $m =n +2$, then  $n_1 = n +2 \equiv 0 \pmod 2$  and $\mu_1  =\frac{n+2}{2}   = \frac{n}{2} +1  \geq  \mu(G) +1$, where the last inequality holds as equality  if and only if $G$ has a 1-factor.

\end{proof}

\begin{lemma}

Let $G$  be a prolific graph on $n$ vertices,  $m$ edges such that $m = n +1$.  Then
\begin{enumerate}
\item{if $n \equiv 1 \pmod 2$, then $\mu_1 \geq \mu+1$ with equality  if and only if  $G$ has a near 1-factor.}
\item{if $n \equiv 0 \pmod 2$, then $\mu_1 \geq \mu(G) +1$  unless $G$ has 1-factor in which case  $\mu_1 = \mu(G)$ but  $\mu_2  \geq \mu +1$.}
\end{enumerate}
\end{lemma}

\begin{proof}
$\mbox{ }$\\
\begin{enumerate}
\item{Suppose $n \equiv 1 \pmod 2$.  Clearly  $n_1 =  m = n +1 \equiv 0 \pmod 2$,  hence $\mu =\frac{n+1}{2}   = \frac{n-1}{2} +1  \geq \mu(G) +1$  where the last inequality holds as equality if and only if $G$ has a near 1-factor.}
\item{Suppose $n \equiv 0 \pmod 2$.  Clearly $n_1 = m \equiv 1 \pmod 2$,  hence $\mu_1 = \frac{n}{2}$.  If $\mu < \frac{n}{2}$ we are done,  otherwise $\mu(G) = \frac{n}{2}$ and $G$ has 1-factor.

By Theorem \ref{avgdeg},  $d_1 \geq 2(d - 1) = 2( \frac{2m}{n} -1)  = 2( \frac{2(n+1)}{n} -1 )  = \frac{2(n+2)}{n}$. Hence $e_1 = \frac{d_1n_1}{2}  \geq  \frac{2(n+1)(n+2)}{2n} > n +2$. Hence  $n_2  = e_1 \geq n+3$ and we deduce that $\mu_2  \geq \frac{ n+2}{2} = \frac{n}{2} +1 > \mu$.}
\end{enumerate}
\end{proof}

\begin{lemma}
Let $G$  be a prolific graph on $n$ vertices and $m$ edges such that $m = n$.  Then
\begin{enumerate}
\item{if $n \equiv 1 \pmod 2$, then $\mu_1 =\frac{n-1}{2} \geq \mu$ with the last inequality holding as equality  if and only if  $G$ has a near 1-factor, and then $\mu_2>\mu_1=\mu$.}
\item{if $n \equiv 0 \pmod 2$ and $G$ has no 1-factor, then $\mu_1 =\frac{n}{2} > \mu(G) $.}
\item{if $n \equiv 0 \pmod 2$ and $G$ has a 1-factor, then $\mu_2 > \mu_1 =\mu =\frac{n}{2}$  unless $G$ is the graph $CP(k,n-k)$ in which case  $\mu_3 >\mu_2= \mu_1=\mu=\frac{n}{2}$.}

\end{enumerate}
\end{lemma}

\begin{proof}
$\mbox{ }$\\
\begin{enumerate}
\item{Suppose $n \equiv 1\pmod 2$. Clearly $n_1  = n$ and $\mu_1 = \frac{n-1}{2} \geq \mu$,  with the last inequality holding as equality if and only  if $G$ has a near 1-factor.

Since $G$ is not 2-regular, we have, by Theorem \ref{avgdeg},   $d_1  >  2(d(G)- 1) = 2$.  Hence $e_1 = \frac{d_1n_1}{2}  > n$ implying $n_2 =  e_1  \geq n +1$ and $\mu_2 \geq \frac{ n+1}{2} > \frac{n-1}{2}  = \mu$.}
\item{Suppose $n \equiv 0 \pmod 2$ and $G$ has no 1-factor.  Then immediately $\mu_1 = \frac{n}{2} > \mu$.}
\item{Suppose $n \equiv 0 \pmod 2$ and $G$ has a 1-factor.  Clearly $\mu_1 = \mu = \frac{n}{2}$.   We observe that since $G$ is prolific, there are no isolated vertices in $G$.  Consider the following two equations:
\begin{align}
&\sum_{j \geq 1} x_j=n=m=n_1 \mbox{ (where $x_j=$  number of vertices of degree $j$)}\\
&\sum_{j \geq 1} jx_j=2n \mbox{ (counting the degrees in $G$)}
\end{align}
From these we get the equations
\begin{align}
& \sum_{j \geq 1} (j-2)x_j=0 \\
&x_1= \sum_{j \geq 2}(j-2)x_j \label{eq4}
\end{align}

Observe, by Theorem \ref{dreg}, that  \[2(e_1 - e(G)) =  \sum_{j \geq  2} j(j-2)x_j  =  - x_1 + \sum_{j \geq  2} j(j-2)x_j  =  \sum_{j \geq 3} (j-1)(j-2)x_j \]  by substituting for $x_1$ using equation \eqref{eq4}.

We consider the following cases:

\smallskip
\noindent \emph{Case 1:}  if for some $j \geq 4$, $x_j > 0$,   then $e_1 - e(G) \geq  3$ hence $n_2 \geq e(G) +3  = n +3$ and $\mu_2  \geq \frac{n}{2} +1 > \mu = \frac{n}{2}$.}
 
\smallskip
\noindent \emph{Case 2:}  if $x_3 \geq 2$ and $x_j = 0$  for $ j \geq 4$, then $e_1 - e(G) \geq  2$ hence $n_2 = e_1 = e(G) +2  = n +2$ and $\mu_2  \geq \frac{n}{2} +1 > \mu = \frac{n}{2}$.

\smallskip
\noindent \emph{Case 3:}  if $ x_3  = 1$  and $x_j =0$  for $j \geq 4$ then by equation \eqref{eq4} above  $x_1 = x_3  =1$ and $x_2 = n-2$.

This forces $e_1- e(G) = 1$ and $e_1 = n+1$ implying that $G$ is  $CP(k,n-k)$ and for $n = 0\pmod 2$,  $CP(k,n-k)$ has indeed  1-factor.  Also $n_2 = e_1  = n +1   = 1\pmod 2$ and  hence $\mu_2  = \frac{n}{2} = \mu_1  = \mu(G)$.

Clearly $L(G)$  is not a  tree or a cycle and we already have $n_2 = e_1 =m+1  = n +1$, and $n_1=m=n$.  Hence, \[d_2  > 2(d_1  -1)  =2 \left (\frac{2e_1}{n_1} -1 \right ) =2 \left (\frac{2(n+1)}{n} -1 \right ) = \frac{2(n+2)}{n}\]  and $e_2 =  \frac{d_2n_2}{2}  > \frac{2(n+2)(n+1)}{2n}  > n+2$  forcing $e_2 \geq n+3 \equiv 1 \pmod 2$ hence $\mu_3 \geq  \frac{n}{2} +1 > \mu(G)$.
\end{enumerate}
\end{proof}

\begin{lemma}
Let $G$  be a prolific graph on $n$ vertices and $m$ edges such that $m = n-1$.  Then $\ind(\mu,G)  \leq 4$ and this is sharp.  
\end{lemma}

\begin{proof}

Clearly $G$ is a tree $T$ which is neither  $K_{1,3}$ nor a path.  Recall Theorem \ref{linetree} which states that  
 \[2e_1(T) -2e(T) = - 2 + \sum_{j\geq 3} (j-1)(j-2)x_j  \]

where $x_j$ is the number of vertices of degree $j$.

We consider two cases  according to  the parity of $n$. 

 \medskip

\noindent \textbf{Case 1} :  $n \equiv 1 \pmod 2$

\smallskip

Clearly if $T$ has no near 1-factor then $\mu  \leq \frac{n-3}{2}$  while $\mu_1 =  \frac{n-1}{2}$  and we are done.

Hence we assume in the sequel  that $T$  has a near 1-factor and $\mu_1=  \mu = \frac{n-1}{2}$.  We consider the following cases:
\begin{enumerate}
\item{Suppose for some $j \geq 4$, $x_j >0$.  Then using Theorem \ref{linetree}  we get $e_1 - e( T)  \geq 2$ and $n_2 = e_1 \geq n+1 \equiv 0 \pmod 2$  hence $\mu_2 \geq \frac{ n+1}{2} > \frac{n-1}{2}  = \mu$.}
\item{Suppose  $x_3 \geq 3$ and $x_j = 0$  for $j \geq 4$.  Then  Theorem \ref{linetree}  we get $e_1 - e( T)  \geq 2$ and $n_2 = e_1 \geq n+1 \equiv 0 \pmod 2$  hence $\mu_2 \geq \frac{ n+1}{2} > \frac{n-1}{2}  = \mu$.}
\item{Suppose  $x_3 =2$ and $x_j = 0$  for $j \geq 4$.  Then  Theorem \ref{linetree}  we get $e_1 - e( T)  =1$ and $n_2 = e_1 =n \equiv 1 \pmod 2$  hence $\mu_2 =\mu_1=\mu = \frac{n-1}{2}$.

But by Theorem \ref{avgdeg},  \[d_2  > 2( d_1  - 1)  =  2(\frac{ 2n}{n-1}  - 1)  = \frac{2(n+1)}{n-1}\] since $L(T)$ is not regular because $d_1 = \frac{2n}{n-1}$ is not an integer  for $n \geq 4$.

Thus we have $e_2 = \frac{d_2n_2}{2} > \frac{2(n+1)n}{2(n-1)}> n+1$ and  hence $n_3  = e_2 \geq n+2$  and  $\mu_3 \geq \frac{n+1}{2} > \frac{n-1}{2} = \mu$.}
\item{Suppose $x_3 = 1$ and $x_j  = 0$  for $j \geq 4$.  It follows that $T$ is a subdivision of $K_{1,3}$.  Clearly $e_1 =e(T) = n -1$  hence  $\mu_2 = \mu_1 = \mu = \frac{n-1}{2}$.  We consider the three types of subdivisions of $K_{1,3}$.

If $T$ is a type C subdivision of $K_{1,3}$ then by Lemma \ref{typeC}, $e_2 = e(T) +3$,  hence $n_3  = n +2$  and $\mu_3 = \frac{ n +1}{2} > \frac{n-1}{2}  = \mu$.

If $T$ is a type B subdivision of $K_{1,3}$ then by Lemma \ref{typeB}, $e_2 = e(T) +2$,  hence $n_3  = n +1 \equiv 0 \pmod 2$  and $\mu_3 = \frac{ n +1}{2} > \frac{n-1}{2}  = \mu$.
 
If $T$ is a type A subdivision of $K_{1,3}$ then by Lemma \ref{typeA}, $e_2 = e(T) +1$,  hence $n_3  = n \equiv 1 \pmod 2$  and $\mu_3= \mu_2 =\mu_1=\mu= \frac{ n -1}{2}$.  But $e_3 \geq e(T) +4$ hence $n_4 \geq n +3 \equiv 0 \pmod 2$  and $\mu_4 \geq \frac{n +3}{2} > \frac{n-1}{2}  = \mu$ and we are done. }
\end{enumerate}

 \medskip

\noindent \textbf{Case 2} :  $n \equiv 0 \pmod 2$

\smallskip

Clearly if $\mu(T) \leq \frac{n-4}{2}$   it follows that $\mu_1 = \frac{n-2}{2}  > \mu$.  So we consider two cases according to $\mu = \frac{n-2}{2}$  or $\mu = \frac{n}{2}$.

\noindent \textbf{Case 2.1} : $\mu = \frac{n-2}{2} =\mu_1$

\begin{enumerate}
\item{Suppose for some $j \geq 4$, $x_j >0$.  Then    we get $e_1 - e( T)  \geq 2$ and $n_2 = e_1 \geq n+1 \equiv 1 \pmod 2$  hence $\mu_2 \geq \frac{ n}{2} > \frac{n-2}{2}  = \mu$.}
\item{Suppose  $x_3 \geq 3$ and $x_j = 0$  for $j \geq 4$.  Then  we get $e_1 - e( T)  \geq 2$ and $n_2 = e_1 \geq n+1 \equiv 1 \pmod 2$  hence $\mu_2 \geq \frac{ n}{2} > \frac{n-2}{2}  = \mu$.}
\item{Suppose  $x_3 =2$ and $x_j = 0$  for $j \geq 4$.  Then  we get $e_1 - e( T)  =1$ and $n_2 = e_1 =n\equiv 0 \pmod 2$  hence $\mu_2 = \frac{n}{2} > \frac{n-2}{2}=\mu$.}
\item{Suppose $x_3 = 1$ and $x_j  = 0$  for $j \geq 4$.  It follows that $T$ is a subdivision of $K_{1,3}$.  Clearly $e_1 =e(T) = n -1$  hence  $\mu_2 = \mu_1 = \mu = \frac{n-2}{2}$.  We consider the three types of subdivisions of $K_{1,3}$.

If $T$ is a type C subdivision of $K_{1,3}$ then by Lemma \ref{typeC}, $e_2 = e(T) +3$,  hence $n_3  = n +2\equiv 0 \pmod 2$  and $\mu_3 = \frac{ n +2}{2} > \frac{n-2}{2}  = \mu$.

If $T$ is a type B subdivision of $K_{1,3}$ then by Lemma \ref{typeB}, $e_2 = e(T) +2$,  hence $n_3  = n +1 \equiv 1 \pmod 2$  and $\mu_3 = \frac{ n }{2} > \frac{n-2}{2}  = \mu$.  
 
If $T$ is a type A subdivision of $K_{1,3}$ then by Lemma \ref{typeA}, $e_2 = e(T) +1$,  hence $n_3  = n \equiv 0 \pmod 2$  and $\mu_3=\frac{ n }{2}> \frac{n-2}{2}  = \mu$.}
\end{enumerate}

\noindent \textbf{Case 2.2} : $\mu = \frac{n}{2}$ and $\mu_1=\frac{n}{2}-1$

\begin{enumerate}
\item{Suppose for some $j\geq 5$,  $x_j > 0$.  Then we get $e_1 - e( T)  \geq 5$ and $n_2  = e_1 \geq n +4  \equiv 0 \pmod 2$  hence  $\mu_2  \geq \frac{n}{2} +2 > \frac{n}{2}  =\mu$.}
\item{Suppose $x_4 \geq 2$.  Then we get  $e_1 - e( T)  \geq 5$ and $ n_2 = e_1 \geq n+4 = 0\pmod 2$  hence  $\mu_2 \geq  \frac{n}{2} +2 > \frac{n}{2}  =\mu$.}
\item{Suppose $x_4  = 1$  and $x_3 \geq 1$.  Then we get $e_1  - e( T)  \geq 3$  and $n_2 = e_1 \geq n+2 \equiv 0 \pmod 2$  hence $\mu_2  \geq  \frac{n}{2} +1 > \frac{n}{2}  =\mu$.}
\item{Suppose $x_4 = 1$ and $x_3 = 0$  and $x_j = 0 $ for $j \geq 5$.  Then $x_1  =  4$  and $x_2 =  n - 5$  and  $T \neq K_{1,4}$  as $n = 0\pmod 2$ so $T $ is a subdivision of $K_{1,4}$.

Using Theorem \ref{linetree} we get $e_1 - e(T) = 2$  hence $n_2 = e_1  = n+1 \equiv 1 \pmod 2$ and $\mu_2 = \frac{n}{2}  = \mu$.  However $L(T)$  is not regular as it must contain a vertex of degree 4 and  a vertex of degree 1.  By Theorem \ref{avgdeg}  $d_2  >  2(d_1 - 1) ) =  2( \frac{2(n+1)}{n-1}- 1 )  = \frac{2(n+3)}{ n-1}$  and $e_2 =  \frac{d_2n_2}{2}  >  \frac{2(n+3)(n+1)}{2(n-1)} > n +3$ hence   $n_3 = e_2 \geq n +4 \equiv 0\pmod 2$  and $\mu_3 \geq \frac{n}{2} +2  > \frac{n}{2} = \mu$.   }
\item{Suppose $x_3 \geq 4 $ and $x_j = 0$  for $j \geq 4$.  Then we get $e_1 - e( T)  \geq 3$  and $n_2  = e_1  \geq n+2 \equiv 0 \pmod 2$  hence  $\mu_2 \geq  \frac{n}{2} + 1 >  \frac{n}{2} = \mu$.}
\item{Suppose $x_3 = 3$ and $x_j = 0$  for $j \geq 4$.  Then  we get  $e_1  - e( T)  = 2$ and $n_2 = e_1 = n+1$  hence $\mu_2 = \frac{n}{2}  =\mu$.  But clearly $L(T)$ is not regular hence  $d_2 > 2( d_1 - 1)  =  2( \frac{2(n+1)}{n-1} - 1)  = \frac{2(n+3)}{n-1}$  and $e_2 =  \frac{d_2n_2}{2}  >  \frac{2(n+3)(n+1)}{2(n-1)} > n +3$ hence   $n_3 = e_2 \geq n +4$, forcing  $\mu_3 \geq \frac{n}{2} +2  > \frac{n}{2} = \mu$.   }
 \item{Suppose  $x_3 = 2$ and $x_j = 0$  for $j \geq 4$.  Then  we get $e_1 - e( T) =1$   and $n_2 = e_1 =  n =0\pmod 2$ and $\mu_2 = \frac{n}{2} = \mu$.  
But $d_2  > 2(d_1 - 1)  =  2( \frac{2n}{n-1}  - 1)  = \frac{2(n+1)}{n-1}$  (since $L(T)$ is not regular because $d_1  = \frac{2n}{n-1}$ is not an integer for $n \geq 4$).

Hence $e_2  = \frac{d_2n_2}{2} > \frac{2(n+1)n}{2(n-1)}> n+1 $ hence $n_3 = e_2  \geq n+2=0\pmod 2$ forcing $\mu_3 \geq \frac{n}{2}+1 >  \frac{n}{2} =  \mu$. }
\item{Suppose $x_3 = 1$ and $x_j  = 0$  for $j \geq 4$.  Then $T$ is a subdivision of $K_{1,3}$.  Clearly we get  $e_1  =e(T) = n -1  \equiv  1\pmod 2$,   hence  $\mu_2 = \mu_1 = \frac{n}{2}  - 1  <  \frac{n}{2} =  \mu$.

If $T$ is a type C subdivision of $K_{1,3}$ then by Lemma \ref{typeC}, $e_2 = e(T) +3$,  hence $n_3  = n +2 \equiv 0 \pmod 2$  and $\mu_3 = \frac{ n +2}{2} > \frac{n}{2}  = \mu$.

If $T$ is a type B subdivision of $K_{1,3}$ then by Lemma \ref{typeB}, $e_2 = e(T) +2$,  hence $n_3  = n +1 \equiv 1 \pmod 2$  and $\mu_3 = \frac{ n }{2}   = \mu$.  But $e_3 > e(T) +7  =  n +6 = 0\pmod 2$ hence $n_4  = e_3  \geq  n+6$ and  $\mu_4 \geq \frac{n}{2} +3 > \frac{n}{2} = \mu$.  We note that $T$ has a 1-factor if and only if $m_1$ and $m_2$ are even, where $m_i$ is the length of path $i$ as per Lemma \ref{typeB}, and $m_1 \geq m_2 > m_3=1$.
 
If $T$ is a type A subdivision of $K_{1,3}$ then  $T$ has no 1-factor contradicting the assumption for Case 2.2.  

}
\end{enumerate} 

\end{proof}

\begin{thmx}
Let $G$ be a prolific graph,  then    $\ind(\mu,G) \leq 4$.  Moreover   $\ind(\mu,G)=4$  if and only if $n\equiv 1 \pmod 2$ and $G$ is a type A subdivision of $K_{1,3}$ or $n \equiv 0 \pmod 2$ and $G$ is a type B subdivision of $K_{1,3}$ with $m_1$ and $m_2$ even.  Hence $k(\mu, \mathcal F)=4$.
\end{thmx}

\subsection{The chromatic number $\chi(G)$:   $k(\chi,\mathcal F)=3$}

For this parameter we use a result by Stacho in \cite{Stacho2001}.  The author defines the following parameter.  Let \[\phi(G) = \max_{u \in V(G)} \max_{\substack{\scriptscriptstyle v \in N(u) \\ \scriptscriptstyle \deg(v) \leq \deg(u)}} \deg(v).\]  The following theorem is then proved:
\begin{theorem} [Stacho] \label{Stacho}
Let $G$ be a graph.  Then $\chi(G) \leq \phi(G)+1$.
\end{theorem}

\begin{observation} \label{vizing}
For a graph $G$, $\chi_1(G)  = \chi'(G) \in \{ \Delta(G), \Delta(G)+1 \}$ by Vizing's Theorem \cite{vizing1964}.
\end{observation}

\begin{thmx} \label{thmchrom}
 Let $G$ be a prolific graph, then $\ind(\chi,G) \leq 3$ and there are infinitely many prolific graphs  for which $\ind(\chi,G)=3$.  Hence $k(\chi,\mathcal F)=3.$

 \end{thmx}

\begin{proof} 
$\mbox{ } \\$
\begin{enumerate}
\item{Suppose $\chi'(G) = \Delta(G) +1$.
 
If $\chi(G)= \Delta(G)$  we are done as, by Observation \ref{vizing} above, $\chi_1(G) > \chi(G)$.

So assume $\chi(G) = \Delta(G) +1$.  Then, since $G$ is prolific, $\Delta \geq 3$,  and by Brook's theorem, $G = K_{\Delta+1}$.  Also, by the chromatic index for complete graphs, $\Delta+1 \equiv 1 \pmod 2$ hence $\Delta \geq 4$ is even.

We can observe that  $L(K_{\Delta+1})$ is $2\Delta - 2$-regular  and hence $L^2(K_{\Delta+1})$ contains a clique on $2\Delta-2$ vertices and $\chi_2(K_{\Delta+1}) \geq 2\Delta - 2 > \Delta+1$  for $\Delta \geq 4$ and we are done.}

\item{Suppose $\chi'(G) = \Delta(G)$.
\begin{enumerate}
\item{if $\chi(G)  < \Delta(G)$ we are done since  $\chi_1(G) > \chi(G)$.}
\item{if $\chi(G) = \Delta(G) +1$ then by Brook's theorem $G = K_{\Delta+1}$ and by the chromatic index of complete graphs, $\Delta+1 \equiv 0 \pmod 2$ hence $\Delta \equiv 1 \pmod 2$  and $\Delta  = 3$ or  $\Delta \geq 5$.  
\begin{itemize}
\item{If $\Delta \geq 5$ then since $L(G)$ is $2\Delta - 2$-regular, $\chi_2(K_{\Delta+1}) \geq 2\Delta - 2 > \Delta+1$  for $\Delta \geq 4$ and we are done.}
\item{If $\Delta = 3$,  then  $G = K_4$, $L(K_4)$ is  4-regular  and by direct checking we get $\chi'_1(K_4) = 4 $, hence $\chi_2(K_4) = 4$  but  $L^2(K_4)$ is 6-regular hence  $\chi_3(K_4) \geq 6  > \chi(K_4)$. }
 \end{itemize}}
\item{Suppose $\chi'(G) =\chi(G) =\Delta(G)$.  If $\deg(u)  = \Delta$ and $w$ is a vertex adjacent to $u$ with $\deg(w) \geq 3$ then the degree of the edge $uw$  in $L(G)$ is $\deg(u) +\deg(w)- 2  \geq \Delta +1$ and $\chi_2(G) \geq \Delta +1$,  and we are done.

So we may assume that the only neighbours of a vertex of degree $\Delta$ are of degrees 1 and 2.

Now by Theorem \ref{Stacho}, $\phi(G)  \leq \max \{ 2, \Delta - 1\}$, since vertices of maximum degree are non-adjacent,  and $\Delta = \chi(G) \leq \phi(G) +1$,   forcing $\phi(G) = \Delta -1$ and hence there exist two adjacent vertices $u$ and $v$  of degree $\Delta -1$ unless $\Delta=3$, in which case it is possible that $\Delta -1 = \deg(u) \leq \deg(v) \leq \Delta$.

So assume $\Delta \geq 4$.  The edge $uv$ forms a vertex of degree $2\Delta -4$ in $L(G)$ hence a clique of order $2\Delta - 4$ in $L^2(G)$ and $\chi_2(G) \geq  2\Delta - 4 >\Delta$ for $\Delta \geq 5$. 

So it remains to  consider the cases $\chi = \Delta \in  \{ 3, 4 \}$.

Let $\chi' = \chi = \Delta = 4$.  Consider a vertex $v$ of  degree 4 in $G$ and let $u_1,\ldots,u_4$ be its neighbours.
\begin{itemize}
\item{If there is an edge say $u_1u_2$ then these five edges form, in $L(G)$, $ K_4$ with a vertex adjacent to the vertices representing the edges  $vu_1$ and $vu_2$  which are  adjacent of degree 4. Hence $\Delta_2(G) \geq 6$ and $\chi_3(G) \geq  6 > \chi(G)= 4$.}
\item{If there is an edge say $wu_1$, then these 5 edges form, in $L(G)$, $ K_4$ with an attached leaf to the vertex representing  $vu_1$. So $L(G)$ contains a vertex of degree 4 adjacent to a vertex of degree at least 3 hence $\Delta_2(G) \geq 5$ and $\chi_3(G) \geq  5 > \chi(G)= 4$.}
\end{itemize}

Let $\chi' = \chi = \Delta = 3$.  Consider a vertex $v$ of  degree 3 in $G$ and let $u_1,u_2,u_3$ be its neighbours.

\begin{itemize}
\item{If there is an edge in the neighborhood of $v$,  say $u_1u_2$, then $G$ contains the graph $H = K_3 +$ attached leaf.  $L(H) = K_4 -e$,  $L^2(H) = W_5$, the wheel  with four vertices and a centre vertex of degree 4  hence $\chi_3(G) \geq \Delta_2(G) \geq 4 > \chi(G)$.      }
\item{If the neighborhood of $v$ has no edges then, since $\chi(G) = \Delta = 3$, $G$ must contains an odd cycle and since $G$ is prolific there is a vertex  $v$ of degree 3 on the odd cycle adjacent to two vertices $u$ and $w$  of degree 2 on the odd cycle.  If $u$ and $w$ are adjacent we again obtain the graph $H = K_3 +$ attached leaf (because of the third edge in $v$) and we are done as before.

Otherwise $u$ and $w$ are nonadjacent and let $u^*$ be adjacent to $u$ and $w^*$ be adjacent to $w$ on the odd cycle.  Observe that if $v$ is adjacent to either $u^*$ or $w^*$  we again have the subgraph $H = K_3 +$attached leaf and we are done.

So the tree on the vertices $u^*,u,v,w,w^*+$ the third edge in $v$  form in $L(G)$  the Bull graph $B$, which is $K_3$ with an attached leaf to two of its vertices,   and in particular we have in $L(G)$ a copy of $K_3 +$attached leaf and we are done as before.} 
\end{itemize}}
  \end{enumerate}}
\end{enumerate}
The infinite family of graphs $CP(3,n-3)$ is such that $\chi(G)  =\chi_1(G) = \chi_2(G)  = 3$  and only $\chi_3(G)  = 4$, hence $k(\chi,\mathcal F)=3$.
\end{proof}

\subsection{The chromatic index $\chi'(G)$:  $k(\chi',\mathcal F)=3$} 
\begin{thmx}

 Let $G$ be a prolific graph, then  $\ind(\chi',G) \leq 3$.  Moreover  $\ind(\chi',G)=3$  if and only if $G = K_{1,4}$, $G$ is a type A subdivision of $K_{1,3}$ or $G$  is obtained by a subdivision of the middle edge of $S_{2,2}$  at least twice.  Hence $k(\chi', \mathcal F)=3$.
\end{thmx} 

\begin{proof}
We consider the following cases:
\begin{enumerate}
\item{$\chi'(G) = \Delta(G)+1$, and  $G$ is in Vizing class 2.  

\noindent

By a result of Fournier \cite{fournier1973colorations}, the graph $G_{\Delta}$, induced in $G$  by the vertices of maximum degree, contains a cycle $C$.  Let $e_1$ and $e_2$ be two incident edges on the cycle $C$.  In $L(G)$,  the vertices representing $e_1$ and $e_2$ are adjacent and have degree $2\Delta - 2$ each.  This shows that $\chi'_1  \geq \Delta_1 =  2\Delta- 2 > \Delta+1$  for $\Delta \geq 4$.

Observe  also that  $\chi'_2 \geq \Delta_2 \geq 2\Delta_1- 2  \geq 4\Delta - 6  > \Delta +1$  for  $\Delta \geq 3$  ($\Delta \geq 3$ forced by $G$ being prolific) which  shows that for every prolific graph of Vizing type 2,  one iteration suffices if $\Delta \geq 4$ and  at most two iterations suffice if $\Delta = 3$.}

\item{$\chi'(G) = \Delta(G)$.  

\noindent
\begin{enumerate}
\item{If $\Delta \geq 5$ then  $L(G)$  contains $K_{\Delta}$ and hence two adjacent vertices of degree $\Delta-1$. So $\chi'_2  \geq \Delta_2  \geq  2\Delta - 4 > \Delta$ for $\Delta \geq 5$, and two iterations suffice.} 
\item{If $\Delta = 4$, consider the vertex $v$ of degree 4. The edges incident with $v$ form $K_4$ in $L(G)$.  We consider two cases:
\begin{itemize}
\item{If $G = K_{1,4}$ then  $\chi' = 4$.  Now $L(G) = K_4$,  $\chi'(K_4) = 3$ and also  $L^2(G)$ is $K_{2,2,2}$ which is  regular of degree 4 and in Vizing class 1. Hence $\chi'_2 = 4$,  and lastly $\chi'_3 \geq \Delta_3  = 6$  and we are done.}
\item{Suppose $G$ is not $K_{1,4}$ and let $v $ be a vertex of degree 4.   Then either there is an edge  between two neighbours of $v$, say $u_1u_2$ is such an edge, or there is an edge incident with say $u_1$.

If there is an edge $u_1u_2$ then  in $L(G)$  we have $K_4$ and an extra vertex represents  $e =u_1u_2$ adjacent to the vertices representing  $vu_1$ and $vu_2$,  and these two vertices are adjacent vertices of degree 4  in $L(G)$.  Hence  $\chi'_2  \geq \Delta_2  \geq  6$,  and we are done.

If there is an edge incident with $u_1$, say $u_1w$,  then, in $L(G)$, we have $K_4$ with one vertex adjacent to the vertex representing the edge $u_1w$ of degree  at least 4  and the other vertices of $K_4$ with degrees at least 3.   Hence  $\chi'_2  \geq \Delta_2  \geq 5$,  and we are done.}
\end{itemize}}
\item{If $\Delta = 3$ then if there are two adjacent vertices of degree 3, $\chi'_1 \geq \Delta_1  = 4$ and we are done.

So consider a vertex $v$ of degree 3 and let $u_1$, $u_2$ and $u_3$ be its neighbours, all of degree at most 2.
\begin{itemize}
\item{Suppose there are two edges,  one edge incident with $u_1$ but not $u_2$ and $u_3$,  and another edge incident with $u_2$ but not with $u_1$ and $u_3$.   Then  $L(G)$  contains two adjacent vertices of degree 3 and $\chi'_2 \geq \Delta_2   = 4$ and we are done.}
\item{Suppose there is an edge incident with $u_1$ and $u_2$.  Again  $L(G)$  contains two adjacent vertices of degree 3 and $\chi'_2 \geq \Delta_2  = 4$ and we are done.}
\item{Suppose every vertex of degree 3 is adjacent with two leaves. Then we assume $G$ is a tree, for if it contains a cycle, since $G$  is prolific there must be a vertex of degree 3 on the cycle  which satisfies one of the cases above and we are done.  However, we proved in Theorem \ref{maxdeg} that in this case,  when we considered the parameter $\Delta$, $\Delta_3 > \Delta$ but  $\Delta_2 = \Delta$ if and only if $G = K_{1,4}$,  $G$ is a type A subdivision of $K_{1,3}$, or $G$  is obtained by subdivision the middle edge of $S_{2,2}$  at least twice.

 For $K_{1,4}$ we already proved that three iterations are necessary and sufficient. For $G$ a type A  subdivision of $K_{1,3}$ we note that both $G$, $L(G)$ and $L^2(G)$  are in Vizing class 1 hence 3 iterations are necessary and sufficient. 

Lastly for the trees obtained by subdividing the middle edge of $S_{2,2}$ at least twice we note  that both $L(G)$  and $L^2(G)$ are in Vizing class 1 and hence  3 iterations are necessary and sufficient, thus completing the proof. }
\end{itemize}}

\end{enumerate}}
\end{enumerate}
\end{proof}

\subsection{The clique number $\omega(G)$:   $k(\omega,\mathcal F)=3$} 

\begin{observation}
For a graph $G$, $\omega_1(G)=\Delta(G)$.
\end{observation}
 
\begin{thmx}
 Let  $G$ be a prolific graph then  $\ind(\omega,G) \leq 3$.  Moreover,  $\ind(\omega,G)=3$   if and only if $\omega= 4$ and $G = K_4 $ or  $\omega=\Delta  = 3$ and the set of vertices of degree 3 forms an independent set in $G$.  Hence $k(\omega,\mathcal F)=3$.
\end{thmx}

\begin{proof}

We consider the following cases:
\begin{enumerate}
 \item{If  $\omega(G) = 2$  ($G$ triangle free) then, since $G$ is prolific, $\Delta \geq 3$  hence $\omega_1 \geq 3 > \omega$.}
\item{If $\omega(G)\geq 5$ then there are adjacent vertices $u$ and $v$ in the clique whose degree is at least $\omega-1$.  The edge $uv$ forms a vertex of degree $2\omega -4$ in $L(G)$ hence $\omega_2 \geq 2\omega - 4 > \omega$ for $\omega \geq 5$.}  
\item{ If $\omega(G) = 4$ then there are two adjacent vertices of degree at least three.  If $G = K_4$  then $\omega(K_4) = 4$ but $\omega_1(K_4)  =\omega (K_{2,2,2}) = 3$  and $L(K_4)$ is 4-regular.  Hence $\omega_2 = 4$ and $L^2(K_4)$ is 6-regular forcing $\omega_3  =6 > \omega$.

If $G$ contains at least 5 vertices then one of the vertices of the 4-clique, say  $v$, is adjacent to a further vertex $w$  not in the clique.  Hence, in the 4-clique there is a vertex  $v$ of degree at least 4 adjacent to a vertex $u$ of degree at least 3.  Therefore $\Delta_1 \geq 5$  and $\omega_2   = \Delta_1 \geq 5 > \omega$.}

 \item{If $\omega(G) = 3$ then
\begin{itemize}
\item{if  $\omega <  \Delta$ we are done since $\omega_1 = \Delta > \omega$.}
\item{if $\omega= \Delta+1$,  then $\Delta = 2$  but $G$ is prolific and $\Delta \geq 3$ and this is impossible.}
\item{if $\omega = \Delta = 3$, then, if two vertices of degree 3  in $G$ are adjacent, $\Delta_1 = 4$ and $\omega_2  = \Delta_1  = 4 > \Delta$. 

If on the other hand, no two vertices of degree 3 are adjacent, then observe that for all  prolific graphs with $\omega =\Delta= 3$ and with no two vertices of degree 3  being adjacent we must have $\Delta_1 = 3$,  $\Delta_2 = 4$ and $\omega_3 =  \Delta_2  = 4 > \Delta_1 = \Delta= \omega = 3$. }
\end{itemize}}
\end{enumerate}
\end{proof}

\subsection{The vertex connectivity $\kappa(G)$ and the edge connectivity $\lambda(G)$} 
 
We will consider these two parameters together because they are very much related.  We use a number of known results which we now list in one main theorem.

\begin{theorem} \label{connectivity}
	For a graph $G$:
	\begin{enumerate}
		\item{\cite{whitney1931theorem} $\kappa(G) \leq \lambda(G) \leq \delta(G)$.}
		\item{ \cite{SHAO20183441} $\lambda_1(G) \geq 2\lambda(G) - 2$. }
		\item{ \cite{SHAO20183441} $\kappa_2(G) \geq 2\kappa(G)-2$. }
		\item{ \cite{chartrand1969connectivity} if $\lambda (G) \geq 2$ then $\kappa_1(G) \geq \lambda(G)$. }
		\item{ \cite{KNOR2003255} if $\delta(G) \geq 3$, then $\kappa_2(G) \geq \delta(G) -1$.}
		\item{\cite{zamtudor}
			\begin{itemize}
				\item{if $\delta_1(G) \leq \lambda(G) \left \lceil \frac{\lambda(G)+1}{2} \right \rceil$, then $\lambda_1(G) \geq \delta_1(G)$.}
				\item{if $\delta_1(G) \geq \lambda(G) \left \lceil \frac{\lambda(G)+1}{2} \right \rceil$, then $\delta_1(G) \geq \lambda_1(G) \geq \lambda(G) \left \lceil \frac{\lambda(G)+1}{2} \right \rceil$.}
		\end{itemize}}
	\end{enumerate}
\end{theorem}

\subsubsection{Minimum degree $\delta\leq 2$:  both connectivities not universal}
 

We simply observe that in the case where $\delta(G) \in \{1,2\}$, we know that the parameter $\delta$ is not universal, so since $\kappa(G) \leq \lambda(G) \leq \delta(G)$, and considering the examples given for the non-universality of $\delta$ in which $\kappa(G) = \lambda(G) = \delta(G)$, we infer that both $\kappa(G)$ and $\lambda(G)$ are non-universal for families of graph in which $\delta \in \{1,2\}$.

\subsubsection{Minimum degree $\delta \geq 3$:   $k(\kappa,\mathcal F)=2$ and $k(\lambda,\mathcal F)=1$} 

\begin{thmx}
	Let $G$ be a prolific graph  with $\delta \geq 3$.  Then $\lambda_1 > \lambda$.
\end{thmx}
\begin{proof}
	We consider the following cases for $\lambda(G)$.
	\begin{enumerate}
		\item{If $\lambda \geq 3$ then by Theorem \ref{connectivity} part 2, $\lambda_1  \geq 2\lambda - 2 >\lambda$  and we are done by one iteration.}
		\item{If $\lambda = 2$  then   $\lambda \left \lceil \frac{\lambda+1}{2} \right \rceil = 4$.  Since  $\delta \geq  3$,  it follows that $\delta_1 \geq 2\delta -2$ hence in both cases of Theorem \ref{connectivity} part 6, we have $\delta_1 \geq 4$ and we are done by one iteration.}
		\item{If $\lambda(G) = 1$  and   $\delta \geq 3$  we show  that  $\lambda_1  \geq 2$ and we are done by one iteration.  Consider $L(G)$ and suppose on the contrary  $\lambda_1  = 1$ implying $L(G)$  has a bridge.  Let $e = uv$ be the bridge in $L(G)$  so that if we remove $e$ then $u$ and $v$ are disconnected.  It follows then that  $u$ and $v$ were incident edges in $G$,  say  $u=xy$ and $v=xz$, forming the bridge $uv$ in $L(G)$.  But since $\delta(G) \geq 3$ there must be another edge say $xw$ incident with the vertex $x$.   
			
			Now, in $L(G)$, the vertex b representing the edge $xw$ in $G$  is adjacent to both $u$ and $v$ so deleting the edge $uv$ does not disconnect $u$ from $v$ as they can reach each other via $b$,  a contradiction.}
	\end{enumerate}
\end{proof}

\begin{thmx} \label{thmconnect}
 For every prolific graph $G$ with $\delta(G)  \geq 3$,  we have $\ind(\kappa,G) \leq2$.  Furthermore, for every $\kappa \geq 1$ and  $\delta \geq 3$, there are graphs with $\ind(\kappa,G) = 2$. Hence  $k(\kappa, \mathcal F) = 2$.
\end{thmx}

\begin{proof}
We consider the following cases:
\begin{enumerate}
\item{If $\kappa \geq 3$  then  by  Theorem \ref{connectivity} part 3, $\kappa_2\geq 2\kappa - 2 > \kappa(G)$  and we are done by two iterations which are necessary (since in fact we can have $\kappa_1 = \lambda(G)$  and we can construct graphs with $\kappa(G) = \lambda(G) = t$  for any value of t).} 
\item{if $\kappa = 2$  then $\lambda \geq 2$ by  Theorem \ref{connectivity} part 1.  However, if $\lambda(G) = 2$ and  $\delta(G)  \geq 3$, we have shown above that $\lambda_1  \geq 4$.  But then  applying Theorem \ref{connectivity} part 4, $\kappa_2  \geq \lambda_1 \geq 4$ and we are done by two iterations. }  
\item{if $\kappa(G) = 1$ then we apply Theorem \ref{connectivity} part 5 so that $\kappa_2 \geq \delta - 1 \geq 2$.  And we are done in two iterations.}

 \end{enumerate}

Chartrand and Harary \cite{chartrand1968graphs} constructed, for every triple of positive integers $1 \leq \kappa \leq \lambda \leq \delta$, graphs with $\kappa(G) = \kappa$, $\lambda(G) = \lambda$ and   minimum degree $\delta(G)$, and these graphs also satisfy  $\kappa_1 = \lambda$ \cite{KNOR2003255}. 

So the Chartrand-Harary graphs for the triples $\kappa = \lambda$ and every $\delta \geq  3$ satisfy  $\kappa_2 > \kappa_1  = \lambda = \kappa$,  completing the proof of the theorem.   
 \end{proof}

\subsection{The independence number $\alpha(G)$}

Recall Theorem \ref{matching} which states that  a connected $K_{1,3}$-free graph on $n$ vertices  has a matching of order $\lfloor \frac{n}{2} \rfloor$.
 
We shall use the following chain of equalities \[\alpha_{k+2}  =  \mu_{k+1} =  \left \lfloor \frac{n_{k+1}}{2 }  \right  \rfloor = \left \lfloor \frac{ e_k}{2 } \right  \rfloor.\]

 \begin{lemma} \label{indep}
Let  $G$ be  a graph on $n$ vertices and $m$ edges with minimum degree $\delta$.  Then $\left \lfloor \frac{ m}{\delta } \right \rfloor \geq \alpha(G)$, with equality if and only if $G$ is  bipartite with one part a maximum independent set with all vertices of degree $\delta$. 

 \end{lemma}
\begin{proof}

Suppose $A$ is a maximum independent set.  Let $B = V \backslash A$.  Let $e(A,B)$ be  the number of edges  from $A$ to $B$.  Then $\delta|A|\leq e(A,B)  \leq e(G)$, and for equality  we need both all vertices in $A$ have degree $\delta$, and also  $B$ to  be  an independent set so that $G$ is bipartite.

 \end{proof}

\begin{thmx} \label{thmindep}

Let $G$ be a prolific  graph on $n \geq 4$ vertices of average degree $d$, minimum degree $\delta$ and independence number $\alpha$.  Then
\begin{enumerate}
\item{if $d \geq 4$ then  $\ind(\alpha,G) \leq 2$.}
\item{if  $\delta \geq 3$ then  $\ind(\alpha,G) \leq 2$.}
\item{if $d \geq 3$  then $\ind(\alpha,G) \leq 3$.}
\item{if $\delta =2$ then $\ind(\alpha,G) \leq 3$.}
\end{enumerate}
\end{thmx}

\begin{proof}
$\mbox{ }$\\
\begin{enumerate}
 \item{Suppose $d \geq 4$.  Then  $n_1 = e(G)  \geq 2n$  and by Theorem \ref{matching} $\alpha_2 = \mu_1 \geq n > n-1  \geq \alpha$.}
\item{Suppose $\delta \geq 3$.  Then by Lemma \ref{indep},  $\alpha \leq  \left \lfloor \frac{ e(G)}{\delta } \right \rfloor \leq \frac{e(G) }{3}$   while $\alpha_2 =  \mu_1 = \left \lfloor \frac{e(G)}{2} \right \rfloor   > \left \lfloor \frac{ e(G)}{3 } \right \rfloor \geq \alpha$  since $\delta \geq 3$.}
\item{Suppose  $d \geq 3$,  then $d_1 \geq 2( d(G) - 1)  \geq 4$  and  \[n_1 = e(G)  = \frac{nd}{2}  \geq \frac{3n }{2}.\]

Hence  $e_1   = \frac{d_1n_1}{2 } \geq 3n$, and we get  \[\alpha_3 =  \mu_2  = \left \lfloor \frac{ n_2}{2} \right \rfloor  = \left \lfloor \frac{ e_1}{2} \right \rfloor \geq \left \lfloor \frac{3n}{2} \right \rfloor  >  n-1 \geq \alpha.\] 

Moreover if $\alpha  < \left \lfloor \frac{  3n}{4 } \right \rfloor$, then already,  from $e(G)  \geq \frac{ 3n}{2}$, we get \[\alpha_2  =  \mu_1  = \left \lfloor \frac{ e(G)}{2 } \right \rfloor \geq  \left \lfloor \frac{ 3n }{4} \right \rfloor > \alpha\] and two iterations suffice.}      
\item{since $\delta= 2$ and $G$ is prolific,  it follows that $\Delta \geq 3$.  Hence  $e(G) \geq  n+1$  and by Lemma \ref{prolific31} part 2 we have $e_1  \geq e(G) +3$.     

 Now since $\delta= 2$, by Lemma \ref{indep} we have $\alpha(G) \leq \frac{e(G)}{2}$  and  $n_2 =  e_1 \geq  e(G) +3$  hence \[\alpha_3 = \mu_2  = \left \lfloor \frac{ e_1}{2} \right \rfloor \geq \left \lfloor \frac{ ( e(G) +3)}{2 } \right \rfloor \geq  \frac{e(G)}{2} +1 \geq \alpha(G) +1 > \alpha.\]}
\end{enumerate}
\end{proof}

\subsection{The domination number $\gamma(G)$}

\begin{thmx} \label{thmdomin}
Let $G$ be a prolific graph.
\begin{enumerate}
\item{If $\delta \geq 4$, then $\ind(\alpha,G) \leq 2$.}
\item{If $\delta=3$, then $\ind(\alpha,G) \leq 3$.}
\item{If  $d \geq 3$, then $\ind(\alpha,G) \leq 3$.}
\end{enumerate}
\end{thmx}
\begin{proof}
$\mbox{ }$\\

Consider $\gamma(G)$ for a graph $G$ on $n$ vertices and its upper bound in terms of $n$.  In \cite{bujtas2016improved} the best upper bounds are summarized in a table according to the value of $\delta$, citing various theorems related to this upper bound.
\begin{enumerate}
\item{For a graph with $\delta \geq 4$,  $\gamma(G)  \leq 0.3637n$.

Now we know that \[\gamma_2=i_2=\mu^*_1 \geq \left \lfloor \frac{n_1}{4} \right \rfloor \geq \left \lfloor \frac{e(G)}{4} \right \rfloor \geq \left \lfloor \frac{2n}{4} \right \rfloor =\frac{n}{2}.\]  So for $\delta \geq 4$, two iterations suffice.}
\item{For $\delta = 3$, $\gamma(G) \leq \frac{3n}{8}$.  A similar argument as above shows that
\[\gamma_2 = i_2  = \mu^*_1 \geq \frac{\mu_1 }{2}  \geq\left \lfloor \frac{n_1}{4} \right \rfloor \geq \left \lfloor \frac{e(G)}{4} \right \rfloor \geq \left \lfloor \frac{3n}{8} \right \rfloor\]  and this might not be enough.  But   

\[\gamma_3 = i_3  = \mu^*_2 \geq \frac{\mu_2}{2}  \geq  \left \lfloor \frac{n_2}{4} \right \rfloor \geq \left \lfloor \frac{e_1(G)}{4} \right \rfloor \geq \left \lfloor \frac{n_1d_1}{8} \right \rfloor \geq  \left \lfloor \frac{4(\frac{3n}{2})}{8} \right \rfloor =\left \lfloor \frac{3n}{4} \right \rfloor> \frac{3n}{8} \geq \gamma(G)\]      and 3 iterations suffice.}
\item{Clearly by Ore's theorem, $\gamma(G)  \leq \frac{n}{2}$.  On the other hand, $d_1 \geq 2(d-1)  \geq 4$ and hence we have  $n_1 = e(G)  = \frac{dn}{2} \geq \frac{3n}{2}$.  Therefore $n_2 = e_1 = \frac{d_1n_1}{2} \geq 3n$  and $\mu_2 = \left \lfloor \frac{n_2}{2 }  \right \rfloor \geq \left \lfloor \frac{3n}{2} \right \rfloor$ and therefore \[\gamma_3  = \mu^*_2 \geq \frac{\mu_2}{2} \geq \frac{3n}{4}  >  \gamma.\]  Hence three iterations suffice.}

\end{enumerate}
\end{proof}


\section{Conclusion}

There are, of course, several open questions which one can obtain by considering the unboundedness or otherwise and the index of parameters other than the fifteen which we have identified in this paper, for example, the largest eigenvalue, which is  unbounded (see \cite{largesteigenvalue,lovasz2007eigenvalues} combined with Theorem 2.3 part 3), the spectral gap, or the size of the automorphism group, to mention only a few. We have chosen fifteen parameters which are very basic in graph theory and whose study in the context of iterated line graphs seems quite natural and interesting.  

For the independence number, determining whether the  index $k(\alpha,\mathcal F) < \infty$, $\mathcal F$ being the family of all prolific graphs with $\delta(G) =1$, is an interesting problem.  Also, to get sharp bounds for $k(\alpha,\mathcal F)$,  where $\mathcal F$ is the family of prolific graphs with respectively $d \geq 4$, $d = 3$, $\delta \geq 3$,  $\delta = 2$ are also interesting tasks.

Similar and even harder problems remains open for the domination and independent domination parameters.  As to characterization of extremal graphs realizing $\ind(P,\mathcal F) = k(P,\mathcal F)$, open problems remains for the parameters chromatic number (Theorem \ref{thmchrom}), vertex connectivity (Theorem \ref{thmconnect}),  together with the already mentioned parameters  independence number  (Theorem  \ref{thmindep}), independence domination number,  and domination number (Theorem \ref{thmdomin}).  

\subsection*{Acknowledgements}

We would like to thank the referees whose careful reading of the paper helped us improve it considerably.
 \bibliographystyle{plain}
\bibliography{iterated bib}

\begin{thebibliography}{10}

\bibitem{ALLAN197873}
R.B. Allan and R.~Laskar.
\newblock On domination and independent domination numbers of a graph.
\newblock {\em Discrete Mathematics}, 23(2):73--76, 1978.

\bibitem{barati}
Z.~Barati.
\newblock Generalized outerplanar index of a graph.
\newblock {\em Czech. Math. Journal}, 68:131--139, 2018.

\bibitem{beineke1970characterizations}
L.~W. Beineke.
\newblock Characterizations of derived graphs.
\newblock {\em Journal of Combinatorial theory}, 9(2):129--135, 1970.

\bibitem{BIEDL20047}
T.~Biedl, E.D. Demaine, C.A. Duncan, R.~Fleischer, and S.G. Kobourov.
\newblock Tight bounds on maximal and maximum matchings.
\newblock {\em Discrete Mathematics}, 285(1):7--15, 2004.

\bibitem{bohme2006linkability}
T.~B{\"o}hme, M.~Knor, and L.~Niepel.
\newblock Linkability in iterated line graphs.
\newblock {\em Discrete mathematics}, 306(7):666--669, 2006.

\bibitem{bujtas2016improved}
C.~Bujt{\'a}s and S.~Klav{\v{z}}ar.
\newblock Improved upper bounds on the domination number of graphs with minimum
  degree at least five.
\newblock {\em Graphs and Combinatorics}, 32(2):511--519, 2016.

\bibitem{caro1980decompositions}
Y.~Caro and J.~Sch{\"o}nheim.
\newblock Decompositions of trees into isomorphic subtrees.
\newblock {\em Ars Combin}, 9:119--130, 1980.

\bibitem{chartrand1968hamiltonian}
G.~Chartrand.
\newblock On hamiltonian line-graphs.
\newblock {\em Transactions of the American Mathematical Society},
  134(3):559--566, 1968.

\bibitem{chartrandHindex}
G.~Chartrand, H.~Gavlas, and M.~Schultz.
\newblock Convergent sequences of iterated $h$-line graphs.
\newblock {\em Discrete Math.}, 147:73--86, 1995.

\bibitem{chartrand1968graphs}
G~Chartrand and F~Harary.
\newblock Graphs with prescribed connectivities.
\newblock {\em Theory of graphs}, pages 61--63, 1968.

\bibitem{chartrand1969connectivity}
G.~Chartrand and M.J. Stewart.
\newblock The connectivity of line-graphs.
\newblock {\em Mathematische Annalen}, 182(3):170--174, 1969.

\bibitem{chartrand1973hamiltonian}
G.~Chartrand and C.E. Wall.
\newblock On the hamiltonian index of a graph.
\newblock {\em Studia Sci. Math. Hungar}, 8:43--48, 1973.

\bibitem{largesteigenvalue}
D.~Cvetkovic and P.~Rowlinson.
\newblock The largest eigenvalue of a graph: A survey.
\newblock {\em Linear and Multilinear Algebra}, 28(1-2):3--33, 1990.

\bibitem{federlimit}
T.~Feder and C.~Subi.
\newblock On the limit average degree of iterated line graphs.
\newblock preprint at theory.stanford.edu.

\bibitem{fournier1973colorations}
J.~Fournier.
\newblock Colorations des ar{\^e}tes d’un graphe.
\newblock {\em Cahiers du CERO (Bruxelles)}, 15:311--314, 1973.

\bibitem{geblehplanarity}
M.~Ghebleh and M~Khatirinejad.
\newblock Planarity of iterated line graphs.
\newblock {\em Discrete Math.}, 308:144--147, 2007.

\bibitem{gutmanchem}
I.~Gutman, \v{Z}. Tomvi\'{c}, B.K. Mishra, and M.~Kuanar.
\newblock on the use of iterated line graphs in quantitative structure-property
  studies.
\newblock {\em Indian J. Chem.}, 40A:4--11, 2001.

\bibitem{hartke1999maximum}
A.W. Hartke, S.G.and~Higgins.
\newblock Maximum degree growth of the iterated line graph.
\newblock {\em the electronic journal of combinatorics}, pages R28--R28, 1999.

\bibitem{hartke2003minimum}
A.W. Hartke, S.G.and~Higgins.
\newblock Minimum degree growth of the iterated line graph.
\newblock {\em Ars Combinatoria}, 69:275--284, 2003.

\bibitem{jensen1906fonctions}
J.L.W.V. Jensen et~al.
\newblock Sur les fonctions convexes et les in{\'e}galit{\'e}s entre les
  valeurs moyennes.
\newblock {\em Acta mathematica}, 30:175--193, 1906.

\bibitem{KNOR2003255}
M.~Knor and L.~Niepel.
\newblock Connectivity of iterated line graphs.
\newblock {\em Discrete Applied Mathematics}, 125(2):255 -- 266, 2003.

\bibitem{knor2006distance}
M.~Knor and L.~Niepel.
\newblock Distance independent domination in iterated line graphs.
\newblock {\em Ars Combinatoria}, 79:161--170, 2006.

\bibitem{knor2006iterated}
M.~Knor and L.~Niepel.
\newblock Iterated line graphs are maximally ordered.
\newblock {\em Journal of Graph Theory}, 52(2):171--180, 2006.

\bibitem{knor2012independence}
M.~Knor and L.~Niepel.
\newblock Independence number in path graphs.
\newblock {\em Computing and Informatics}, 23(2):179--187, 2012.

\bibitem{knorchem}
M.~Knor, P.~Poto\"{c}nik, and R.~\"{S}krekovski.
\newblock The {W}eiener index in iterated line graphs.
\newblock {\em Discrete Appl. Math.}, 160:2234--2245, 2020.

\bibitem{kotzig1957theory}
A.~Kotzig.
\newblock From the theory of finite regular graphs of degree three and four.
\newblock {\em {\^C}asopis Pestov. Mat}, 82:76--92, 1957.

\bibitem{krausz1943demonstration}
J.~Krausz.
\newblock D{\'e}monstration nouvelle d’une th{\'e}oreme de whitney sur les
  r{\'e}seaux.
\newblock {\em Mat. Fiz. Lapok}, 50(1):75--85, 1943.

\bibitem{lovasz2007eigenvalues}
L.~Lov{\'a}sz.
\newblock Eigenvalues of graphs.
\newblock 2007.

\bibitem{SHAO20183441}
Y.~Shao.
\newblock Essential edge connectivity of line graphs.
\newblock {\em Discrete Mathematics}, 341(12):3441 -- 3446, 2018.

\bibitem{Stacho2001}
L.~Stacho.
\newblock New upper bounds for the chromatic number of a graph.
\newblock {\em Journal of Graph Theory}, 36(2):117--120, 2001.

\bibitem{sumner1974graphs}
D.P. Sumner.
\newblock Graphs with 1-factors.
\newblock {\em Proceedings of the American Mathematical Society}, 42(1):8--12,
  1974.

\bibitem{van2020minimum}
W.C. van Batenburg.
\newblock Minimum maximal matchings in cubic graphs.
\newblock {\em arXiv preprint arXiv:2008.01863}, 2020.

\bibitem{van1965interchange}
A.C.M. van Rooij and H.S. Wilf.
\newblock The interchange graph of a finite graph.
\newblock {\em Acta Mathematica Academiae Scientiarum Hungarica},
  16(3):263--269, 1965.

\bibitem{vizing1964}
V.G. Vizing.
\newblock On an estimate of the chromatic class of a p-graph.
\newblock {\em Discret Analiz}, 3:25--30, 1964.

\bibitem{wang}
J.~Wang.
\newblock Nonplanarity of iterated line graphs.
\newblock {\em J. of Mathematics}, 2020:Article ID 5752806, 2020.

\bibitem{west2001introduction}
D.B. West et~al.
\newblock {\em Introduction to graph theory}, volume~2.
\newblock Prentice hall Upper Saddle River, 2001.

\bibitem{whitney1931theorem}
H.~Whitney.
\newblock A theorem on graphs.
\newblock {\em Annals of Mathematics}, pages 378--390, 1931.

\bibitem{zamtudor}
T.~Zamfirescu.
\newblock On the line-connectivity of line-graphs.
\newblock {\em Mathematische Annalen}, 187:305--309, 12 1970.

\end{thebibliography}
\end{document}